\documentclass[a4paper,oneside,11pt,final]{article}

\usepackage[T1]{fontenc}
\usepackage[utf8]{inputenc}
\usepackage{hyperref} 
\usepackage[english]{babel}
\usepackage{enumerate}
\usepackage{aliascnt}
\usepackage{amssymb}
\usepackage{amsmath}
\usepackage{amsthm}
\usepackage{verbatim}
\usepackage[normalem]{ulem}
\usepackage{mathtools}
\usepackage{esint} 

\usepackage[noabbrev,capitalize]{cleveref}

\usepackage[right]{showlabels}

\usepackage[nottoc]{tocbibind} 
\usepackage[title]{appendix}
\usepackage{xcolor}

\numberwithin{equation}{section} 

\def\todo{\marginpar{\textcolor{red}{todo}}\textcolor{red}}

\usepackage{hyperref} 

\usepackage{autonum} 

\allowdisplaybreaks[4] 

\theoremstyle{plain}

\newtheorem{proposition}{Proposition}[section]

\newaliascnt{lemma}{proposition} 

\newtheorem{lemma}[lemma]{Lemma}

\aliascntresetthe{lemma} 
\Crefname{lemma}{Lemma}{Lemmas}

\newaliascnt{theorem}{proposition} 
\newtheorem{theorem}[theorem]{Theorem}

\aliascntresetthe{theorem} 

\newaliascnt{corollary}{proposition} 
\newtheorem{corollary}[corollary]{Corollary}

\aliascntresetthe{corollary}

\newaliascnt{hypothesis}{proposition}

\aliascntresetthe{hypothesis} 

\theoremstyle{definition}

\newaliascnt{definition}{proposition} 
\newtheorem{definition}[definition]{Definition}

\aliascntresetthe{definition} 
\Crefname{definition}{Definition}{Definitions}

\newaliascnt{problem}{proposition} 

\aliascntresetthe{problem} 

\newaliascnt{example}{proposition} 

\aliascntresetthe{example} 

\theoremstyle{remark}

\newaliascnt{remark}{proposition} 
\newtheorem{remark}[remark]{Remark}

\aliascntresetthe{remark} 

\def\equationautorefname~#1\null{%
	(#1)\null
}

\newcommand{\R}{\mathbb{R}}
\newcommand{\C}{\mathbb{C}}

\newcommand{\N}{\mathbb{N}}
\newcommand{\E}{\mathcal{E}}

\renewcommand{\S}{\mathbb{S}}

\newcommand{\B}{\mathcal{B}}

\newcommand{\W}{\mathcal{W}}
\newcommand{\tr}{\mathrm{tr}}
\newcommand{\M}{\mathcal{M}}

\newcommand*{\dd}{\mathop{}\!\mathrm{d}}
\newcommand{\Imag}{\operatorname{im}}

\title{
    Stability of the free boundary Willmore problem 
}

\author{
    \large{Anna Dall'Acqua}\thanks{Institute for Applied Analysis, Ulm University, Helmholtzstraße 18, 89081 Ulm, Germany. \texttt{anna.dallacqua@uni-ulm.de}} 
    \and 
    \large{Fabian Rupp}\thanks{Faculty of Mathematics, University of Vienna, Oskar-Morgenstern-Platz 1, 1090 Vienna, Austria. \texttt{fabian.rupp@univie.ac.at}} 
    \and 
    \large{Reiner Schätzle}\thanks{Fachbereich Mathematik, Eberhard-Karls-Universität Tübingen, Auf der Morgenstelle 10, 72076 Tübingen, Germany. \texttt{schaetz@everest.mathematik.uni-tuebingen.de}} 
    \and 
    \large{Manuel Schlierf}\thanks{Mathematics department, Salzburg University, Hellbrunner Str. 34, 5020 Salzburg, Austria. \texttt{math@manuelschlierf.info}}
}

\date{}

\begin{document}
\maketitle
\vspace{-0.5cm}
\begin{abstract}
    We study the Willmore problem with free boundary by means of a new 
    {\L}ojasiewicz--Simon gradient inequality for functionals on infinite dimensional manifolds. In contrast to previous works, we do not rely on a gradient-like representation of the Fr\'echet derivative, but merely on an inequality. 
    For the free boundary Willmore flow, we prove that solutions starting sufficiently close to a local minimizer exist for all times and converge. In the static setting, we prove quantitative stability of free boundary Willmore immersions and a local rigidity result in a neighborhood of free boundary minimal surfaces.
\end{abstract}

{\small

}

\section{Introduction and main results}

Let $\Sigma$ be a compact oriented surface with boundary $\partial\Sigma$ and let $f\colon \Sigma\to\R^3$ be an immersion. The \emph{Willmore energy} of $f$ is defined as
\begin{equation}
    \W(f)=\frac14\int_{\Sigma} H^2\dd\mu,
\end{equation}
where $\mu$ is the Riemannian measure on $\Sigma$ induced by the pull-back $g=f^*\langle\cdot,\cdot\rangle$ of the Euclidean metric in $\R^3$. Moreover, in local positive coordinates $A_{ij}=\langle\partial_{ij}^2f,\nu\rangle$ is the second fundamental form where 
\begin{align}\label{eq:def_normal}
\nu=\frac{\partial_1f\times\partial_2f}{|\partial_1f\times\partial_2f|}
\end{align}
is our choice for the unit normal and $H=g^{ij}\langle\partial_{ij}^2f,\nu\rangle$ is the mean curvature. While the  functional $\mathcal{W}$ was already studied by Blaschke and Thomsen \cite{blaschke}, who proved its conformal invariance for closed surfaces, it became more popular by the work of Willmore \cite{MR0202066}, who observed that round spheres are the global minimizers among closed surfaces. The first variation of the Willmore energy (for variations supported away from $\partial\Sigma$) can be represented by integration against its $L^2$-gradient, given by
\begin{equation}
    \nabla\W(f) =\frac12(\Delta_g H+|A^0|^2H)\nu,
\end{equation}
where $\Delta_g$ is the Laplace--Beltrami operator on $(\Sigma,g)$ and $A^0=A-\frac12 Hg$ is the trace-free second fundamental form. On closed surfaces, the associated dynamic problem, the \emph{Willmore flow}, is the geometric evolution equation
\begin{align}
\partial_tf = -2 \nabla \mathcal{W}(f) = - (\Delta_g H+|A^0|^2H)\nu,
\end{align}
for immersions $f\colon[0,T)\times\Sigma \to\R^3$. This is a fourth order quasilinear parabolic PDE, so no maximum principles are available. While short-time existence and uniqueness of smooth solutions can be established by standard contraction arguments, the global behavior of the flow and its potential singularities are much more delicate and remain open problems, in general.

The first contribution for the Willmore flow of closed surfaces is \cite{simonett2001}, where global existence and convergence for initial data close to the sphere have been established. In a series of papers \cite{kuwertschaetzle2002,kuwertschaetzle2001,kuwertschaetzle2004}, a curvature-concentration based blow-up procedure was developed that was eventually used to prove global existence and convergence for spherical initial data $f_0\colon \S^2\to\R^3$ satisfying $\W(f_0)\leq 8\pi$. This was recently extended to tori with rotational symmetry \cite{dallacquamullerschatzlespener2020}. In both situations, the $8\pi$-threshold that arises from the Li--Yau inequality \cite{liyau1982} is sharp \cite{blatt2009,dallacquamullerschatzlespener2020}, but the precise nature of the singularities, in particular the question if they occur in finite or infinite time, remains largely unexplored, despite some recent progress \cite{MR3331730,dallacqua2025dimensionreductionwillmoreflows,mäderbaumdicker2025willmoreenergylandscapespheres}. Lately, also the Willmore flow with Dirichlet boundary conditions has been studied in \cite{schlierf2024,eichmann2024,MR4884019} under rotational symmetry assumptions. Again, an energy threshold related to the Li--Yau inequality ensures global existence and convergence.

One of the few arguments ensuring global existence and convergence of the Willmore flow  that does not involve the Li--Yau threshold is based on the \emph{{\L}ojasiewicz--Simon inequality} \cite{MR160856,MR727703}. 
For the Willmore functional, this inequality has been proven in \cite{chillfasangovaschaetzle2009} relying on the general sufficient conditions derived in \cite{chill2003}. A rather general argument then ensures that the Willmore flow with initial data close to local minimizer exists for all times and converges to a local minimum. 

The contribution of this work is twofold. First, we revisit the problem of finding sufficient conditions for the {\L}ojasiewicz--Simon inequality to be satisfied for an abstract energy functional. We identify such conditions for functionals on Banach manifolds, a framework which is suitable for the discussion of non-linear boundary problems, such as \emph{free boundary problems} of arbitrary order.
Second, we apply this result to \emph{Willmore immersions with free boundary}, resulting not only in asymptotic and quantitative stability of local minimizers, but also in a local rigidity result for free boundary minimal surfaces.

\subsection{The {\L}ojasiewicz--Simon gradient inequality on Banach manifolds}

Functions satisfying a {\L}ojasiewicz--Simon gradient inequality possess a remarkable behavior in a neighborhood of their critical points. In finite dimensions, the inequality is satisfied by any \emph{(real) analytic} energy by {\L}ojasiewicz's result \cite{MR160856}, whereas analyticity alone is never enough in infinite dimensional Banach spaces \cite{rupp2020}. Combining \cite{MR160856} with a Lyapunov--Schmidt reduction, Simon \cite{MR727703} was first to prove the inequality in the context of infinite dimensional problems. 
In modern works, the validity of the {\L}ojasiewicz--Simon gradient inequality is usually ensured by studying the second variation. In (geometric) applications one usually requires the second variation to be a Fredholm operator of index zero or even invertible (modulo invariances); see, for instance, \cite{chill2003,ChillISEM,MR4129355,rupp2020} for general results in Banach spaces and \cite{chillfasangovaschaetzle2009,MR2594589,MR3352243,MR4451897,MR4315554,MR4523489,MR4726523,MR4855302,maederbaumdicker2025quantitativeestimatesrelativeisoperimetric,döhrer2025convergencegradientflowsknotted} for some recent applications in the context of geometric flows. A particular challenge in higher order evolutions is that identifying the kernel of the second variation is generically more difficult than in second order problems.

The setup has been extended to incorporate \emph{constraints}; see \cite{rupp2020} for finite codimension constraints in general, and also \cite{okabe2023convergencesobolevgradienttrajectories,döhrer2025convergencegradientflowsknotted} for specific examples of infinite codimension.

Our first main result is the following, see \Cref{sec:banach_manifolds} for a review on Banach manifolds.

\begin{theorem}\label{intro-thm:loja-abstract}
    Let $V$, $Z$ be two Banach spaces, $\mathcal{M}$ an analytic $V$-Banach manifold, $u_0\in U_0\subset \mathcal{M}$ an open subset, and $\E\colon U_0\subset\mathcal{M}\to\R$, $\delta\E\colon U_0\subset\mathcal{M}\to Z$ be two analytic maps satisfying
    \begin{enumerate}[(i)]
        \item\label{item:loja_abstract_1} $\|\E'(u)\|_{(T_{u}\mathcal{M})^*} \leq M \|\delta\E(u)\|_Z$ for $u\in U_0$, where the norm on $T_{u}\mathcal{M}$ is induced by a local tangent bundle chart and $M<\infty$,
        \item\label{item:loja_abstract_2} $\mathrm{ker}(\delta\E)'(u_0)\subset T_{u_0}\mathcal{M}$ is finite dimensional, and
        \item\label{item:loja_abstract_3} $\mathrm{im}(\delta\E)'(u_0)=\vcentcolon W$ is closed and complemented in $Z$.
    \end{enumerate}
    Then $\E$ satisfies a \L ojasiewicz-Simon gradient inequality on $\M$ close to $u_0$, that is, there exist $\theta\in (0,\frac12]$, $C<\infty$, and an open neighborhood $U_0(u_0)\subset U_0\subset\mathcal{M}$ such that
    \begin{equation}\label{eq:loja-abstract}
        |\E(u)-\E(u_0)|^{1-\theta} \leq C \|\delta \E(u)\|_Z\quad\text{for all $u\in U_0(u_0)$.}
    \end{equation}
\end{theorem}

In view of previous abstract sufficient conditions for the {\L}ojasiewicz--Simon inequality, this result is an extension in various directions. First, it allows for functionals defined on manifolds, extending \cite{rupp2020}, where only level set submanifolds of finite codimension were considered. Second, conditions \eqref{item:loja_abstract_2} and \eqref{item:loja_abstract_3} are satisfied whenever $(\delta \E)'(u_0)$ is Fredholm (of any index), but are strictly weaker assumptions. A third and main novelty is that the connection between  $\delta \E$ and the Fr\'echet derivative $\E'$ is only in terms of the norm inequality in \eqref{item:loja_abstract_1}. Usually, in previous works, cf.\ \cite{chill2003,ChillISEM,MR4129355,rupp2020}, $\delta \E$ is assumed to be a gradient representation of $\E'$ with respect to some type of inner product.

For the proof of \Cref{intro-thm:loja-abstract}, we first give a diffeomorphism
invariant version of the {\L}ojasiewicz--Simon gradient inequality,
see \Cref{thm:loja-abstract-linear} and \Cref{rem:loja}  below, where the manifold
$\mathcal{M}$ is just a Banach space. This is obtained by revisiting
the argument in \cite{chill2003}, see also \cite{MR727703}. Then \Cref{intro-thm:loja-abstract} follows
by choosing local charts.

\Cref{intro-thm:loja-abstract} provides the right general framework to apply the {\L}ojasiewicz--Simon gradient inequality to problems involving general nonlinear boundary conditions or even a 
free boundary. Indeed, the Banach manifold framework is a natural structure on the set of admissible functions satisfying certain (free) boundary conditions. Our second main result illustrates this for the Willmore flow with free boundary.

\subsection{Willmore immersions and the Willmore flow with free boundary}

Let $S\subset\R^3$ be an embedded surface and denote by $N^S\colon S\to\S^2\subset\R^3$ a normal field along $S$. Moreover, let $\Sigma$ be a compact, oriented surface with boundary $\partial\Sigma\neq \emptyset$.
An immersion $f\colon \Sigma\to\R^3$ is said to \emph{meet $S$ orthogonally} along $\partial\Sigma$ if
\begin{equation}\label{intro-eq:geombc}
    f(\partial \Sigma)\subset S\quad \text{and} \quad \langle \nu_f,N^S\circ f\rangle\Big|_{\partial \Sigma} = 0.
\end{equation}

A critical point of the Willmore energy in the class of immersions with \eqref{intro-eq:geombc} satisfies an additional \emph{natural} boundary condition. This boundary term has been computed in \cite[Equations~(2.19) and (2.20)]{alessandronikuwert2016}. For the purpose of this paper, we work with the following.

\begin{definition} A smooth immersion $f\colon \Sigma \to \R^3$ is said to satisfy the \emph{free boundary conditions} on $S$ if 
\begin{equation}\label{intro-eq:fbc}
    \begin{cases}
        f(\partial \Sigma)\subset S,\\
        \langle \nu_f,N^S\circ f\rangle\Big|_{\partial \Sigma} = 0,\\
        \frac{\partial H}{\partial \eta} + A^S(\nu,\nu) H \Big|_{\partial \Sigma} = 0
    \end{cases}
\end{equation}
where  $A^S$ denotes the (scalar) second fundamental form of $S$ and $\eta=\eta_f$ is the outer conormal along $\partial \Sigma$ induced by $g_f = f^*\langle\cdot,\cdot\rangle$.
\end{definition}

We now state the \L ojasiewciz--Simon inequality for the Willmore energy on immersions with free boundary on $S$.

\begin{theorem}\label{intro-thm:main-result-fbc}
    Let $S$ be analytic and let $\Sigma$ be a compact oriented surface with non-empty boundary and $\bar{f}\colon \Sigma\to\R^3$ be a smooth immersion satisfying \eqref{intro-eq:fbc}. Then there exist $\theta\in(0,\frac12]$, $C>0$, and $\sigma>0$ such that, for $f\in W^{4,2}(\Sigma,\R^3)$ with $\|f-\bar{f}\|_{W^{4,2}(\Sigma)}<\sigma$ satisfying the free boundary conditions \eqref{intro-eq:fbc} in the trace sense, 
    \begin{equation}\label{intro-eq:Willmore_LS}
        |\W(f)-\W(\bar{f})|^{1-\theta} \leq C \|\nabla\W(f)\|_{L^2(\dd\mu_f)}.
    \end{equation}
\end{theorem}
We obtain the {\L}ojasiewicz--Simon gradient inequality in $W^{4,2}$ and
with the norm of the $L^2$-gradient on the right-hand side of \eqref{intro-eq:Willmore_LS},
corresponding to $Z = L^2$ in \Cref{intro-thm:loja-abstract}, as we follow the argument
in \cite{chillfasangovaschaetzle2009,rupp2023}. One may be tempted to believe that a stronger inequality
should be valid for the Willmore functional on a space closer to its
natural energy space of $W^{2,2} \cap W^{1,\infty}$-immersions
and closer to $Z = (W^{2,2} \cap W^{1,\infty})^*$. In case such an inequality
is true, its proof would require new ideas beyond \cite{chillfasangovaschaetzle2009,rupp2023}.

To prove \Cref{intro-thm:main-result-fbc}, we first construct suitable \emph{generalized Gauss coordinates} in a neighborhood of an immersion $\bar f\colon \Sigma\to\R^3$ satisfying \eqref{intro-eq:fbc}. Our approach is inspired by the work of Stahl \cite{stahl1996} where short-time existence of the mean curvature flow with free boundary was studied. The main idea is to parametrize a tubular neighborhood of a fixed immersion $\bar f$ satisfying \eqref{intro-eq:fbc} using the vector flow with infinitesimal generator given by an extension of the Gauss map $\nu_{\bar f}$. 

In our application, to apply \Cref{intro-thm:loja-abstract}, we need to ensure analyticity of the Willmore energy on immersions given in such a way. We thus need to consider a more explicit construction as in \cite{stahl1996} and require $S$ to be analytic, allowing us to prove that \eqref{intro-eq:fbc} determines an analytic Banach manifold. Then, to verify \eqref{item:loja_abstract_2} and \eqref{item:loja_abstract_3}, we restrict to variations which are parametrized in generalized Gauss coordinates, and then apply (elliptic) PDE theory. 
For \eqref{item:loja_abstract_1}, 
 we crucially rely on the fact that the additional third order boundary condition in \eqref{intro-eq:fbc} arises as the natural boundary condition. We may then apply \Cref{intro-thm:loja-abstract} to the Willmore functional in generalized Gauss coordinates which then proves \Cref{intro-thm:main-result-fbc}.
 
An important class of global free boundary Willmore minimizers are given by \emph{free boundary minimal surfaces.} 
Indeed, such surfaces meet the support surface $S$ orthogonally, while trivially being Willmore minimizing as $H\equiv 0$, in particular, they satisfy \eqref{intro-eq:fbc}. Inside the unit ball $B^3\subset \R^3$ (so $S=\S^2$), it has been proven by Nitsche \cite{MR784101} that the only free boundary minimal disks are the equatorial ones. Such a rigidity result is not known for surfaces with larger genus or number of boundary components. 
In the work of Fraser--Schoen \cite{MR3461367}, it has been shown that for any $\ell\in\N$ there exists a smooth embedded free boundary minimal surface in $B^3$ with genus zero and $\ell$ boundary components; see also \cite{MR4524829}, the recent survey \cite{MR4251121}, and the references therein.

As a first application of \Cref{intro-thm:main-result-fbc}, we obtain a local rigidity statement for the Willmore problem with free boundary. We term an immersion $f\colon\Sigma\to\R^3$ satisfying \eqref{intro-eq:fbc} and $\nabla \W(f)=0$ a \emph{free boundary Willmore immersion}. Our next result shows that, in a neighborhood of a free boundary minimal surface, any free boundary Willmore immersion must also be already minimal. This applies in particular to the minimal surfaces in $B^3$ obtained in \cite{MR3461367}.

\begin{corollary}
    Let $\Sigma$ be a compact surface with $\partial\Sigma\neq \emptyset$ and let $\Omega\subset \R^3$ be an open set with analytic boundary $S=\partial\Omega$.    
    Let $\bar f\colon \Sigma\to \R^3$ be an immersed free boundary minimal surface in $\Omega$. Then there exists $\sigma>0$ such that any free boundary Willmore immersion $f\colon \Sigma \to \R^3$ with $\Vert f-\bar f\Vert_{W^{4,2}(\Sigma;\R^3)}< \sigma$ is also minimal, that is, $H_f\equiv 0$.
\end{corollary}
\begin{proof}
Let $\sigma>0$ be the size of the {\L}ojasiewicz neighborhood of $\bar f$. Then, we have
\begin{align}
    |\mathcal{W}(f)-\mathcal{W}(\bar f)|^{1-\theta} \leq C \Vert  \nabla \mathcal{W}(f)\Vert_{L^2(\dd \mu_f)} =0,
\end{align}
since $f$ is a free boundary Willmore immersion. Thus, $\mathcal{W}(f) =\mathcal{W}(\bar f)=0$ as $\bar f$ is minimal.
\end{proof}

The second application of \Cref{intro-thm:main-result-fbc} is on the associated gradient flow.
A family of immersions $f\colon[0,T)\times \Sigma\to\R^3$ satisfying
\begin{equation}\label{intro-eq:free-bdry-willmore-flow}
    \begin{cases}
        \langle \partial_t f,\nu\rangle = -\big(\Delta_gH+|A^0|^2H\big)&\text{on $[0,T)\times \Sigma$}\\
        f(0,\cdot)=f_0&\text{on $\Sigma$}\\
        f(t,\partial \Sigma)\subset S &\text{for $0\leq t<T$}\\
        \langle \nu, N^S\circ f\rangle = 0 &\text{on $[0,T)\times\partial \Sigma$}\\
        \frac{\partial H}{\partial\eta} + A^S(\nu,\nu)H = 0&\text{on $[0,T)\times\partial \Sigma$}
    \end{cases}  
\end{equation}
is called \emph{free boundary Willmore flow} starting in $f_0$. Here we assume that the initial datum $f_0\colon \Sigma \to \R^3$  is a smooth immersion satisfying also \eqref{intro-eq:fbc}.

The \emph{free boundary Willmore flow} can be seen as a higher-order analog of the mean curvature flow with free boundary conditions in the sense that stationary solutions for the mean curvature flow are also stationary solutions for the Willmore flow. In literature, there are various results on the free boundary mean curvature flow. In particular, the flow is studied for graphs in \cite{huisken89}, well-posedness in the general case is treated in \cite{stahl1996}, a level-set formulation is studied in \cite{gigaohnumasato99} and a notion of Brakke flow is given in \cite{edelen2020}. Less is known for its higher-order analog given by \Cref{intro-eq:free-bdry-willmore-flow}. In \cite{metsch2022}, a constrained version of \eqref{intro-eq:free-bdry-willmore-flow} was studied where a non-local Lagrange multiplier guarantees fixed surface area $\lambda>0$ along the evolution. Assuming disk-type topology and that $\lambda$ is sufficiently small, \cite{metsch2022} yields global existence and sub-convergence to critical points of the area-preserving flow.

In this work, we establish the following local stability result for the unconstrained free boundary Willmore flow by adapting the arguments in \cite{chillfasangovaschaetzle2009} and applying \Cref{intro-thm:loja-abstract}.

\begin{theorem}\label{intro-thm:FBWF}
    Let $\alpha\in (0,1)$, $S$ be analytic, $\Sigma$ be a compact oriented surface with non-empty boundary, and let $\bar f\colon \Sigma\to\R^3$ be a local $C^4$-minimizer of the Willmore energy among immersions with \eqref{intro-eq:fbc}. Then there exists $\bar \varepsilon=\bar \varepsilon(\alpha,\bar f)>0$ such that, for all  immersions $f_0\in C^{4,\alpha}(\Sigma,\R^3)$ satisfying \eqref{intro-eq:fbc} and 
    \begin{equation}
        \|f_0-\bar f\|_{C^{4,\alpha}(\Sigma)} < \bar \varepsilon,
    \end{equation}
    there exists a global free boundary Willmore flow starting in $f_0$ (up to reparametrization) which converges smoothly to a free boundary Willmore surface $f_{\infty}$ with $\W(f_\infty)=\W(\bar f)$ for $t\to\infty$.
\end{theorem}

A simple example of a disk-type local Willmore minimizer with free boundary which is not minimal is a hemisphere sitting on a plane $S\subset \R^3$. The fact that this is the minimizer follows from a reflection argument and \cite{MR0202066}.

Moreover, as a corollary of our flow analysis, we obtain a local quantitative stability estimate controlling the distance from the set of free boundary Willmore immersions by the Willmore energy deficit. Similar strategies have also been used in \cite{MR4568345,rupflin2023sharpquantitativerigidityresults,MR4896163,maederbaumdicker2025quantitativeestimatesrelativeisoperimetric}.

\begin{corollary}\label{cor:quanti_stability}
Let $S, \Sigma,\bar f, \bar\varepsilon,\alpha$ and $f_0$ be as in \Cref{intro-thm:FBWF} and let 
\begin{align}
\mathcal{C}(\Sigma,S) \vcentcolon = \big\{f\colon \Sigma\to\R^3\mid f\text{  satisfies \eqref{intro-eq:fbc} and $\nabla\mathcal{W}(f)=0$}\big\}.
\end{align}
Then there exist $C=C(\alpha,\bar f)\in (0,\infty)$ and $\gamma = \gamma(\alpha,\bar f)\in (0,\frac12]$ such that
\begin{align}
\inf \{ \Vert f_0\circ\Psi-f\Vert_{C^4(\Sigma)} \mid f\in \mathcal{C}(\Sigma,S), \Psi\colon \Sigma\to\Sigma \text{ $C^4$-diffeomorphism}\} \leq C \big(\mathcal{W}(f_0)-\mathcal{W}(\bar f)\big)^{\gamma}.
\end{align}
\end{corollary}

We prove \Cref{intro-thm:loja-abstract} at the end of \Cref{sec:abstr-theory}, \Cref{intro-thm:main-result-fbc} at the end of \Cref{sec:loja_Willmore}, and \Cref{intro-thm:FBWF} and \Cref{cor:quanti_stability} at the end of \Cref{sec:freebdy}.

\section{The \L ojasiewicz-Simon inequality on Banach manifolds}\label{sec:abstr-theory}

To investigate {\L}ojasiewicz--Simon inequalities for functionals in nonlinear spaces, we translate the problem to the linear setting through local coordinates. The goal is to find sufficient conditions in the linear setting that behave well with respect to transformations by diffeomorphism.

Consider an energy functional $\E$ on a Banach space $V$ and a diffeomorphism $\varphi$ between open subsets of $V$. The chain rule yields that $\tilde \E \vcentcolon = \E\circ\varphi$ satisfies
\begin{align}\label{eq:pizza}
\tilde\E'(\tilde v) = \E'(\varphi(\tilde v)). \varphi'(\tilde v),\qquad \tilde v\in V.
\end{align}
Classical approaches to the {\L}ojasiewicz--Simon inequality involve the second variation. Equation \eqref{eq:pizza} shows that to translate properties from $\E''$ to $\tilde \E''$, one must control the influence of $\varphi''$. This approach was followed in \cite{rupp2020}, relying on previous work in the linear setting in \cite{chill2003}. In the case of submanifolds, the resulting inequality then involves the tangential gradient. 

We now present a novel result in Banach spaces that behaves well under diffeomorphisms as above without requiring further assumptions on $\varphi''$; see \Cref{rem:loja} below for details.

\begin{theorem}\label{thm:loja-abstract-linear}
Let $V$, $Z$ be two Banach spaces, $v_0\in V_0\subset V$ an open subset, and let $\mathcal{E}\colon V_0\to\R$, $\delta \mathcal{E}\colon V_0\to Z$ be two analytic maps satisfying
    \begin{enumerate}[(i)]
        \item\label{item:loja_lin_1} $\|\mathcal{E}'(v)\|_{V^*} \leq M \|\delta\mathcal{E}(v)\|_Z$ for $v\in V_0$ and some $M<\infty$,
        \item\label{item:loja_lin_2} $\mathrm{ker}(\delta\mathcal{E})'(v_0)\subset V$ is finite dimensional, and
        \item\label{item:loja_lin_3} $\mathrm{im}(\delta\mathcal{E})'(v_0)=\vcentcolon W$ is closed and complemented in $Z$.
    \end{enumerate}
    Then $\mathcal{E}$ satisfies a \L ojasiewicz-Simon gradient inequality close to $v_0$, that is, there exist $\theta\in (0,\frac12]$, $C<\infty$, and an open neighborhood $V_0(v_0)\subset V_0\subset V$ such that
    \begin{equation}\label{eq:loja-abstract}
        |\mathcal{E}(v)-\mathcal{E}(v_0)|^{1-\theta} \leq C \|\delta \mathcal{E}(v)\|_Z\quad\text{for all $u\in V_0(v_0)$.}
    \end{equation}
\end{theorem}

Usually, see for instance \cite[Theorem 1.2]{rupp2020}, condition \eqref{item:loja_lin_1} is verified when the derivative $\E'$ can be represented by $\delta\E\colon V_0\subset V\to Z^*$ via continuous, linear maps $i_v\in \mathcal{L}(V,Z)$ for $v\in V_0$ in the form
\begin{align}\label{eq:rem_loja_1}
&\E'(v).\xi = (\delta \E)(v).(i_v\xi) \quad \text{ for }\xi\in V,\\
&\text{where } \Vert i_v\Vert_{\mathcal{L}(V,Z)} \leq M
\end{align}
for some $M<\infty$. Here $Z^*$ plays the role of $Z$ above. Clearly, \eqref{eq:rem_loja_1} implies \eqref{item:loja_lin_1}. Conditions \eqref{item:loja_lin_2} and \eqref{item:loja_lin_3} are satisfied if $(\delta\E)'(v_0)\colon V\to Z$ is a Fredholm operator of any (finite) index.

The conditions in \Cref{thm:loja-abstract-linear} are now well-behaved with respect to diffeomorphisms.
\begin{remark}\label{rem:loja}
\label{item:rem_loja_diffeo} For an analytic diffeomorphism, $\phi\colon \tilde{V}_0\subset V\overset{\approx}{\to}V_0$, we put $\tilde\E \vcentcolon = \E \circ \phi\colon \tilde V_0\to\R$, $\widetilde{\delta \E}\vcentcolon = (\delta \E)\circ\phi \colon \tilde V_0\to Z$ and see
\begin{align}
\tilde \E'(\tilde v) = \E'(\phi(\tilde v))\circ D\phi(\tilde v),
\end{align}
hence, since $\Vert D\phi(\tilde v)\Vert_{\mathcal{L}(V,V)}\leq C_\phi$ for all $\tilde v\in \tilde V_0$, after passing to a smaller $\tilde V_0$ if necessary, we find
\begin{align}
    \Vert \tilde \E'(\tilde v)\Vert_{V^*}\leq \Vert \E'(\phi(\tilde v)) \Vert_{V^*} \Vert D\phi(\tilde v)\Vert_{\mathcal{L}(V,V)} \leq C_\phi M \Vert \widetilde{\delta \E}(\tilde v)\Vert_Z,
\end{align}
which gives \eqref{item:loja_lin_1} for $\tilde \E$ and $\widetilde{\delta \E}$. Next, for $\tilde v_0 \vcentcolon = \phi^{-1}(v_0)$, we have
\begin{align}
    \widetilde{\delta\E}'(\tilde v_0) = (\delta \E)'(v_0)\circ D\phi(\tilde v_0).
\end{align}
As $D\phi(\tilde v_0)\colon V \overset{\approx}{\to}V$ is a linear isomorphism, we get
\begin{align}
&\ker \widetilde{\delta\E}'(\tilde v_0) = D \phi(\tilde v_0)^{-1} \ker (\delta \E)'(v_0),\\
&\Imag \widetilde{\delta
\E}'(\tilde v_0) = \Imag
(\delta\E)'(v_0) = W\subset Z,
\end{align}
which gives \eqref{item:loja_lin_2} and \eqref{item:loja_lin_3} for $\tilde
\E$ and $\widetilde{\delta\E}$. This shows the assumptions \eqref{item:loja_lin_1}, \eqref{item:loja_lin_2}, and \eqref{item:loja_lin_3} in \Cref{thm:loja-abstract-linear} remain true under an (analytic) diffeomorphic transformation.
\end{remark}

\begin{proof}[Proof of \Cref{thm:loja-abstract-linear}]
We follow \cite{chill2003} and denote by $C\in (0,\infty)$ a constant depending on $\E, \delta\E,M$ which is allowed to change from line to line. Without loss of generality, we assume $v_0=0$ and that $\delta \E(0) =0$, as otherwise  the inequality is immediate by continuity. 

By \eqref{item:loja_lin_2}, the kernel $X\vcentcolon = \ker (\delta\E)'(0) \subset V$ is finite dimensional, so there exists $Y\subset V$ closed such that $V=X\oplus Y$, cf.\ \cite[XV, Corollary 1.6]{lang1993}. Let $X(0),Y(0)$ be open neighborhoods of $0$ in $X$, $Y$ with $X(0)\times Y(0)\subset V_0$.  By \eqref{item:loja_lin_3}, there exists a continuous linear projection $P_W\colon Z\to W$, cf.\ \cite[XV, Corollary 1.6]{lang1993}.
Consider the analytic map
\begin{align}
\Phi\colon X(0)\times Y(0)\times W \to W, \quad\Phi(x,y,w)\vcentcolon = P_W(\delta \E(x,y)) - w.
\end{align}
By construction, we have that
\begin{align}
    \partial_y \Phi(0,0,0)= (\delta\E)'(0)\vert_Y \colon Y\to W 
\end{align}
is an isomorphism. Since $\Phi(0,0,0)= \delta\E(0)=0$, the implicit function theorem \cite[XIV,
Theorem 2.1]{lang1993} yields that, there exist sufficiently small neighborhoods $X(0), Y(0), W(0)$ in $X,Y,W$, and an analytic map $g\colon X(0)\times W(0)\to Y(0)$ such that
\begin{align}\label{eq:loja_lin_1}
    P_W (\delta \E(x,g(x,w))) = w \quad\text{ for all } x\in X(0),w\in W(0).
\end{align}
The analytic map $\E_0\colon X(0)\to\R, \E_0(x) \vcentcolon= \E(x,g(x,0))$ is defined on an open subset of a finite dimensional Banach space. By the finite dimensional {\L}ojasiewicz gradient inequality \cite{MR160856}, after passing to a smaller neighborhood $X(0)$, there exist $\theta\in (0,\frac{1}{2}]$, $C>0$ such that
\begin{align}
    |\E_0(x)-\E_0(0)|^{1-\theta} \leq C |\E'_0(x)| \quad \text{ for all }x\in X(0).
\end{align}
 For a basis $\{e_1,\dots,e_d\}$ of $X$, $\Vert e_i\Vert_V=1$, $i=1,\dots,d$, we compute
\begin{align}
|\partial_{e_i} \E_0(x) | &= |\E'(x,g(x,0)) . (e_i+\partial_{e_i} g(x,0))| \\
&\leq \Vert \E'(x,g(x,0))\Vert_{V^*} \Vert e_i + \partial_{e_i} g(x,0)\Vert_V \leq C\Vert \E'(x,g(x,0))\Vert_{V^*},
\end{align}
again passing to smaller $X(0)$ if necessary. Hence by the above
\begin{align}
    |\E(x,g(x,0)) - \E(0)|^{1-\theta}\leq C \Vert \E'(x,g(x,0))\Vert_{V^*}\quad \text{ for all }x\in X(0).\label{eq:E_0_Loja}
\end{align}
For $(x,y)$ close to $0$, we see that $w\vcentcolon = P_W(\delta\E)(x,y)$ is close to $0$ and, by \eqref{eq:loja_lin_1}, we have
\begin{align}
P_W(\delta\E)(x,g(x,w)) = w = P_W(\delta\E)(x,y).
\end{align}
By uniqueness of $g(x,w)$ for $x,y,w$ close to $0$, we get $y=g(x,w)$, and estimate
\begin{align}
\Vert y-g(x,0)\Vert_V= \Vert g(x,w)-g(x,0)\Vert_Y \leq C  \Vert w \Vert_W &= C\Vert P_W (\delta\E)(x,y)\Vert_W \\
&\leq C\Vert \delta \E(x,y)\Vert_Z.\label{eq:loja_lin_2}
\end{align}
By Taylor expansion of $\E$, we obtain
\begin{align}
&|\E(x,y)-\E(x,g(x,0))| = |\E'(x,y).(0, y-g(x,0)) + O(\Vert y-g(x,0)\Vert_V^2)|\\
&\quad \leq \Vert \E'(x,y)\Vert_{V^*} \Vert y-g(x,0)\Vert_V + C\Vert y-g(x,0)\Vert_V^2 \\
&\quad \leq C\Vert \delta\E(x,y)\Vert_Z^2,
\end{align}
using \eqref{item:loja_lin_1} in the last step. Combining with \eqref{eq:E_0_Loja}, we find
\begin{align}
|\E(x,y)-\E(0)|^{1-\theta} &\leq |\E(x,y)-\E(x,g(x,0))|^{1-\theta} + |\E(x,g(x,0))-\E(0)|^{1-\theta} \\
&\leq C \Vert \delta \E(x,y)\Vert_Z^{2(1-\theta)}+ C \Vert \E'(x,g(x,0))\Vert_{V^*}.
\end{align}
By the mean value theorem, \eqref{item:loja_lin_1}, and \eqref{eq:loja_lin_2}, 
for $x,y$ close to $0$, we get
\begin{align}
    \Vert \E'(x,g(x,0))\Vert_{V^*} &\leq \Vert \E'(x,y)\Vert_{V^*} + C\Vert y-g(x,0)\Vert_V\\
    &\leq C\Vert \delta\E(x,y)\Vert_Z.
\end{align}
Since $\theta\in (0,\frac12]$, for $x,y$ close to $0$, we conclude
\begin{align}
|\E(x,y)-\E(0)|^{1-\theta}\leq C\Vert \delta \E(x,y)\Vert_Z. &\qedhere
\end{align}
\end{proof}

The nonlinear version for functionals defined on open subsets of Banach manifolds is \Cref{intro-thm:loja-abstract} which we may now prove.

\begin{proof}[Proof of \Cref{intro-thm:loja-abstract}]
Let $\varphi\colon V_0\subset V\overset{\approx}{\to} U_0$ be an analytic inverse chart for $U_0\subset \mathcal{M}$ in a possibly smaller open neighborhood of $u_0$ with $V_0\subset V$ open. We verify that $\tilde \E\vcentcolon = \E \circ\varphi \colon V_0\to\R$ and $\widetilde{\delta \E}\vcentcolon = (\delta \E)\circ\varphi \colon V_0 \to Z$ satisfy conditions \eqref{item:loja_lin_1}, \eqref{item:loja_lin_2}, and \eqref{item:loja_lin_3} of \Cref{thm:loja-abstract-linear} in $v_0\vcentcolon = \phi^{-1}(u_0)\in V_0$. Indeed
\begin{align}
    \tilde \E'(v) = \E'(\varphi(v)) \circ D\varphi(v),
\end{align}
hence, for $v\in V_0$, 
after possibly reducing $V_0$, we have
\begin{align}
    \Vert \tilde \E'(v) \Vert_{V^*} & \leq \Vert \E'(\varphi(v))\Vert_{T_{\varphi(v)}\mathcal{M}^*} \Vert D\varphi(v)\Vert_{\mathcal{L}(V, T_{\varphi(v)}\mathcal{M})} \\
    &\leq C_\varphi M \Vert\delta \E(\varphi(v))\Vert_Z = C_\varphi M \Vert \widetilde{\delta\E}(v)\Vert_Z,
\end{align}
which gives \eqref{item:loja_lin_1} for $\tilde \E$ and $\widetilde{\delta \E}$. As
\begin{align}
(\widetilde{\delta \E})'(v_0) = (\delta\E)'(u_0)\circ D\varphi(v_0),
\end{align}
and since $D\varphi(v_0)\colon V\overset{\approx}{\to}T_{\varphi(v_0)}\mathcal{M}$ is a linear isomorphism, we get
\begin{align}
&\ker \widetilde{\delta\E}'(\tilde v_0) = D \phi(\tilde v_0)^{-1} \ker (\delta \E)'(v_0),\\
&\Imag \widetilde{\delta
\E}'(\tilde v_0) = \Imag
(\delta\E)'(v_0) = W\subset Z,
\end{align}
which gives \eqref{item:loja_lin_2} and \eqref{item:loja_lin_3} for $\tilde\E$ and $\widetilde{\delta\E}$.
\end{proof}

\section{The \L ojasiewicz--Simon  inequality for the Willmore energy of free boundary surfaces}\label{sec:loja_Willmore}

In the following, we fix $\Sigma$ a compact, connected, orientable, smooth surface with non-empty boundary $\partial\Sigma$.
Consider an analytic, orientable, and embedded surface $S$ in $\R^3$ without boundary. 

On $\Sigma$ we consider the pull-back metric $g_f=f^*\langle\cdot,\cdot\rangle$. For $x \in \partial \Sigma$, we work with a basis $\{\tau(x), \eta(x)\}$ of the tangent space $T_x \Sigma$ that is orthonormal (with respect to $g_f(x)$) and such that $\tau(x)$ is tangential to $\partial \Sigma$ at $x$ and $\eta(x)$ is the outer co-normal. If an immersion $f\colon \Sigma \to\R^3$ meets $S$ orthogonally along $\partial\Sigma$, cf.\ \eqref{intro-eq:fbc}, then
\begin{equation}
    \langle \frac{\partial f}{ \partial  \tau}, N^S \circ f \rangle\Big|_{\partial \Sigma}=0\quad\text{ and }\quad N^S \circ f = \pm \frac{\partial f}{ \partial  \eta}.
\end{equation}
In the following we choose the normal field $N^S$ such that  
\begin{equation} \label{eq:choiceNS}
    N^S \circ f = \frac{\partial f}{ \partial  \eta} .
\end{equation}

\begin{remark}\label{rem:geomS}
Since $S \subset \R^3$ is real analytic, by \cite[Theorem 1 in Sect. 2.12.3]{Simon1996} the maps
\begin{equation} S \ni y \to N^S(y) \quad \mbox{ and } 
\Psi_S\colon S \times \R \to \R^3, \quad \Psi_S(y,\rho)\vcentcolon=y+\rho N^S(y),
\end{equation} 
are then real analytic. We wish to restrict $\Psi_S$ to a tubular neighborhood. Since $S$ is not assumed to be compact, the tubular neighborhood does not necessarily have a uniform radius, in general. On the other hand, since in the rest of the paper we work with immersions whose image is contained in a fixed compact subset of $\R^3$, we can assume without loss of generality that such a uniform radius $\varepsilon_S>0$ does exist. In particular, 
\begin{equation}\label{eq:PsiS}
\Psi_S\colon S\times(-\varepsilon_S,\varepsilon_S)\to\mathcal{U}_S=\Psi_S(S\times(-\varepsilon_S,\varepsilon_S))
\end{equation}
is a diffeomorphism. Writing $(\Pi^S,d^S)\vcentcolon=\Psi_S^{-1}\colon\mathcal{U}_S\to S\times(-\varepsilon_S,\varepsilon_S)$, we say that $\Pi^S$ is the nearest point projection onto $S$ and 
\begin{align}
  \mathcal{U}_S \ni y \mapsto d^S(y) \in \R\label{eq:sign_dist},
\end{align}
is the (signed) distance function to $S$. Both of these maps are real-analytic.

We recall that for $y \in \mathcal{U}_S$, $v , w\in \R^3$
\begin{align}
D\Pi^S (y).v& = v - \langle v, N^S(\Pi^S(y)) \rangle N^S(\Pi^S(y))\\
\label{eq:abldistance}
        \nabla d^S (y)& =N^S(\Pi^S(y))  \mbox{ and}\\
        A^S_y(v,w) & = - \langle v, D(N^S \circ \Pi^S)(y) . w \rangle = - \langle v, \nabla^2 d^S(y). w \rangle, 
    \end{align}
    where we usually omit the dependence on $y$ and we have extended $A^S$ to $\mathcal U_S\times \R^3\times \R^3$ in the usual way.
\end{remark}

\begin{remark}\label{rem:dpartialsigma}
    Using Nash's embedding theorem, assume that $\Sigma\subset\R^M$ is a submanifold for some $M\in\N$ and fix $g_0=\iota^*\langle\cdot,\cdot\rangle$,  the pull-back with respect to the inclusion $\iota\colon\Sigma\to\R^M$. 
    Suppose that $r_0>0$ is sufficiently small such that the intrinsic nearest point projection $B_{r_0}(\partial\Sigma)\vcentcolon = \{ x\in \Sigma \mid d^{\partial\Sigma}_0(x)<r_0\}\to\partial\Sigma$ with respect to $g_0$ is well-defined and smooth. In particular, $d^{\partial \Sigma}_0\colon\Sigma\to[0,\infty)$, the intrinsic distance to $\partial\Sigma\subset\Sigma$ with respect to $(\Sigma,g_0)$, is smooth on $B_{r_0}(\partial\Sigma)$.
    
    Moreover, denote by $\mu_{g_0}$ the Riemannian measure with respect to the reference metric $g_0$ and by $\frac{\dd\mu}{\dd\mu_{g_0}}$ the Radon--Nikodym-density of any absolutely continuous measure $\mu$ on $\Sigma$ with respect to $\mu_{g_0}$.

    Furthermore, let $\eta_0\colon\partial\Sigma\to\S^{M-1}\subset\R^M$ denote the outer unit conormal with respect to $g_0$ and fix a unit tangent field $\tau_0$ of $\partial\Sigma$ such that $\{\eta_0(x),\tau_0(x)\}$ is an oriented $g_0$-orthonormal basis of $T_x\Sigma$ for all $x\in\partial\Sigma$. Let $R$ be the open set of $C^1$-Riemannian metrics on $\Sigma$. For any $g\in R$, we can derive the following formula for the outer unit co-normal $\eta_g$: 
    \begin{equation}\label{eq:eta_g-expl-form}
        \eta_g = \frac{g(\tau_0,\tau_0)\eta_0-g(\eta_0,\tau_0)\tau_0}{\sqrt{g(\eta_0,\eta_0)g(\tau_0,\tau_0)-g(\eta_0,\tau_0)^2}\sqrt{g(\tau_0,\tau_0)}}.
    \end{equation}
\end{remark}

\subsection{Existence of generalized Gaussian coordinates}

Fix now a smooth immersion $\bar{f}\colon\Sigma\to\R^3$ of $\Sigma$ satisfying \eqref{intro-eq:fbc}. 
Fix a family of open $\Sigma_j\subset\Sigma$ with $\Sigma=\bigcup_{j=1}^K\Sigma_j$ such that $\bar f|_{\Sigma_j}$ is an embedding. In this section, we prove the following

\begin{proposition}\label{prop:ex-gauss-coord}
    There exist $\alpha_0,\bar r>0$ and a smooth $\Phi\colon \Sigma\times(-\bar r,\bar r)\to \R^3$ such that each $\Phi|_{\Sigma_j\times(-\bar r,\bar r)}\colon\Sigma_j\times(-\bar r,\bar r)\to\mathcal U_j\vcentcolon=\Phi(\Sigma_j\times(-\bar r,\bar r))$ is a diffeomorphism
    and 
    \begin{enumerate}[(i)]
        \item $\Phi(x,0)=\bar{f}(x)$ for all $x\in  \Sigma$. 
        \item\label{item:xinormal} Denote $\xi\colon \Sigma\times(-\bar r,\bar r)\to \R^3$, $\xi \vcentcolon = \partial_r \Phi$. Then, we have 
        \begin{equation}\label{eq:controlxi}
        \xi(x,0)=\nu_{\bar{f}}(x) \mbox{ on } \partial \Sigma, \quad \quad \frac12\leq \langle\xi(\cdot,0),\nu_{\bar{f}}\rangle\leq 2 \mbox{ on }\Sigma.
        \end{equation}
        \item \label{item:PhivaluesS} $\Phi(x,r)\in S$ for all $x\in\partial \Sigma$, $r\in(-\bar{r},\bar{r})$, and $\Phi(x,r)\in \mathcal{U}_S$ for all $x\in B_{\alpha_0}(\partial\Sigma)$ and $r\in (-\bar r, \bar r)$;
        \item\label{item:ex-gauss-coord_3}        
        For all local inverse charts $\psi\colon\Omega\overset{\approx}{\to}\psi(\Omega)\subset\Sigma$, $\Omega\subset \R\times[0,\infty)$ (relatively) open, $\psi\in C^\infty(\bar \Omega)$, and any $\beta \in \N_0^2$, $(x,r)\mapsto \partial_x^{\beta}(\Phi(\psi(x),r))$ is analytic in $r$, uniformly with respect to $x\in\Omega$, see \Cref{sec:analycompo}. The same is true for $\partial_x^{\beta}\partial_r(\Phi(\psi(x),r))$.
    \end{enumerate}
\end{proposition}

In order to ensure that $\Phi(x,r) \in S$ for $x \in \partial \Sigma$, we need to adapt the usual parametrization of the tubular neighborhoods of $\bar f(\Sigma_j)$. Further, in view of the \L ojasiewicz-Simon inequality \Cref{intro-thm:loja-abstract}, we need to give special attention to ensuring analyticity of the mappings involved in the construction. Consider
\begin{equation}\label{eq:def-Phi-D}
\Phi_{\bar f}\colon \Sigma \times \R \to \R^3, \quad  \Phi_{\bar f}(x,r)\vcentcolon=\bar{f}(x)+r\nu_{\bar{f}}(x)  ,
\end{equation}
with $\nu_{\bar{f}}$ the normal field of $\bar{f}$.
Then, there exists $\bar\varepsilon>0$ such that 
\begin{equation}
\Phi_{\bar f}\colon \Sigma_j\times(-\bar\varepsilon,\bar\varepsilon)\to \mathcal{U}_{\bar f,j}\vcentcolon =\Phi_{\bar f}(\Sigma_j\times(-\bar\varepsilon,\bar\varepsilon))
\end{equation}
is a diffeomorphism for all $1\leq j\leq K$. Using Remark \ref{rem:geomS}, for $(x,r)\in \Phi_{\bar f}^{-1}(\mathcal{U}_S)$ with $\bar{f}(x)\in\mathcal{U}_S$, define
\begin{equation}
    \Phi_S(x,r)\vcentcolon=\Pi^S(\Phi_{\bar f}(x,r)) + d^S(\bar{f}(x)) N^S(\Pi^S(\Phi_{\bar f}(x,r))).
\end{equation}
Using that $\bar{f}$ satisfies \eqref{intro-eq:fbc}, there exists $\Tilde{\varepsilon}_S\in(0,r_0)$ with $r_0$ as in \Cref{rem:dpartialsigma} satisfying 
\begin{equation}\label{eq:etildaS}
    \bar{f}(B_{\tilde \varepsilon_S}(\partial\Sigma))\subset\mathcal{U}_S   .
\end{equation}
Consider arbitrary $\alpha_0>0$ with
$0<2\alpha_0<\tilde\varepsilon_S$. Due to the choice of $\Tilde{\varepsilon}_S$,  
and by the uniform continuity of $\Phi_{\bar f}$, there exists $\tilde r_{\bar f}(\alpha_0)\in (0,\bar\varepsilon)$ satisfying
\begin{equation}\label{eq:choice-rD}
  B_{2\alpha_0}(\partial\Sigma)\times (-\tilde r_{\bar f}(\alpha_0),\tilde r_{\bar f}(\alpha_0)) \subset \Phi_{\bar f}^{-1}(\mathcal{U}_S).
\end{equation}
In particular, with $r=0$, we have $\Phi_{\bar f}(x,0)=\bar{f}(x)\in\mathcal U_S$ for all $x\in B_{2\alpha_0}(\partial \Sigma)$.

For $\alpha_0$ as above, let $\zeta\in C^\infty(\R)$ be nonincreasing with
\begin{equation}\label{eq:def-zeta}
    \chi_{(-\infty,\alpha_0)}\leq \zeta\leq \chi_{(-\infty,2\alpha_0)}
    \quad\text{and}\quad |\zeta'|\leq \frac{2}{\alpha_0}\quad\text{on $\R$}.
\end{equation}Altogether, we define $\Phi\colon \Sigma\times (-\tilde r_{\bar f}(\alpha_0),\tilde r_{\bar f}(\alpha_0))\to\R^3$,
\begin{equation}\label{eq:def-Phi}
    \Phi(x,r)\vcentcolon=(1-\zeta(d^{\partial\Sigma}_0(x)))\Phi_{\bar f}(x,r) + \zeta(d^{\partial\Sigma}_0(x))\Phi_S(x,r).
\end{equation}
Then $\Phi$ is well-defined, smooth, and satisfies
\begin{align}
    \Phi &= \Phi_S \mbox{ in } B_{\alpha_0}(\partial\Sigma)\times(-\tilde r_{\bar f}(\alpha_0),\tilde r_{\bar f}(\alpha_0)), \\
    \Phi&= \Phi_{\bar f}\mbox{ in } \big(\Sigma\setminus B_{2\alpha_0}(\partial\Sigma)\big)\times(-\tilde r_{\bar f}(\alpha_0),\tilde r_{\bar f}(\alpha_0)).\label{eq:Phi-bdry-values}
\end{align}

We now show the restrictions $\Phi|_{\Sigma_j\times(-\bar r,\bar r)}\colon \Sigma_j\times(-\bar r,\bar r)\to\mathcal{U}_j$ are diffeomorphisms for suitable $\alpha_0$ and $\bar r$. The first step is

\begin{lemma}\label{lem:C1-perturb-est-Phi}
    In the above construction, for any $\varepsilon \in (0,\frac12)$, there exists $\alpha_0$ with $0<2\alpha_0<\tilde\varepsilon_S$ and $\bar r\in (0,\tilde r_{\bar f}(\alpha_0))$ such that
    \begin{equation}
        \|\Phi-\Phi_{\bar f} \|_{C^1(\Sigma\times(-\bar r,\bar r))} < \varepsilon.
    \end{equation}
\end{lemma}
\begin{proof}
    We fix the value of $\alpha_0$ later in the proof. First, we observe that $\Phi_{\bar f}(x,r) \in \mathcal{U}_S$ for $x\in B_{2\alpha_0}(\partial\Sigma)$, $|r|< \tilde{r}_{\bar f}(\alpha_0)$ by \eqref{eq:choice-rD} and hence, with \Cref{rem:geomS},
    \begin{align}
        \Phi_{\bar f}(x,r) &= \Psi_S (\Pi^S(\Phi_{\bar f}(x,r)),d^S(\Phi_{\bar f}(x,r))) \\ 
        &= \Pi^S(\Phi_{\bar f}(x,r)) + d^S(\Phi_{\bar f}(x,r)) N^S(\Pi^S(\Phi_{\bar f}(x,r))).\label{eq:detprojS}
    \end{align}
    Therefore, for $x\in \Sigma$ and $|r|<\tilde r_{\bar f}(\alpha_0)$,
    \begin{align}
        \Phi(x,r)&=\Phi_{\bar f}(x,r) + \zeta(d^{\partial\Sigma}_0(x))\big(\Phi_S(x,r)-\Phi_{\bar f}(x,r)\big)\\
        &=\Phi_{\bar f}(x,r) + \zeta(d^{\partial\Sigma}_0(x))\big(d^S(\bar{f}(x))-d^S(\Phi_{\bar f}(x,r))\big)N^S(\Pi^S(\Phi_{\bar f}(x,r))).\label{eq:dec-of-Phi}
    \end{align}
    In view of \eqref{eq:dec-of-Phi}, we first show that there exist $\alpha_0$ and $\bar r''<\tilde r_{\bar f}(\alpha_0)$ such that
    \begin{equation}
        \|d^S\circ \bar{f}-d^S\circ \Phi_{\bar f}\|_{C^1(B_{2\alpha_0}(\partial\Sigma)\times(-\bar r'',\bar r''))} < C \varepsilon\label{eq:C1-perturb-est-Phi}
    \end{equation}
    for some universal constant $C$. Using \eqref{intro-eq:fbc},
    \begin{equation}
        \langle N^S\circ\Pi^S\circ\Phi_{\bar f},\nu_{\bar{f}}\rangle \Big|_{(x,r)\in \partial \Sigma\times\{0\}} = \langle N^S\circ\bar{f},\nu_{\bar{f}}\rangle \Big|_{\partial \Sigma} = 0.
    \end{equation}
    By uniform continuity, there exist $\alpha_0\in(0,\frac12\Tilde{\varepsilon}_S)$ and $\bar r'\in (0,\min\{\bar \varepsilon,\tilde r_{\bar f}(\alpha_0)\})$ such that
    \begin{equation}\label{eq:choice-of-rD'}
        \|\langle N^S\circ\Pi^S\circ\Phi_{\bar f},\nu_{\bar{f}}\rangle\|_{C^0(B_{2\alpha_0}(\partial\Sigma)\times(-\bar{r}',\bar{r}'))} < \varepsilon.
    \end{equation}
    By the fundamental theorem of calculus and \eqref{eq:abldistance}, for $(x,r) \in B_{2\alpha_0}(\partial \Sigma) \times (-\bar{r}',\bar{r}')$
    \begin{align}
        d^S&(\bar{f}(x))-d^S(\Phi_{\bar f}(x,r)) = d^S(\bar{f}(x)) - d^S(\bar{f}(x)+r\nu_{\bar{f}}(x))  \\
        &= -\int_0^r \frac{\dd }{\dd \rho} d^S(\bar{f}(x)+\rho\nu_{\bar{f}}(x))\dd \rho = -\int_0^r\langle N^S(\Pi^S(\Phi_{\bar f}(x,\rho))),\nu_{\bar{f}}(x)\rangle \dd\rho\label{eq:ftc}
    \end{align}
    So, clearly, 
    \begin{equation}\label{eq:est-dsbarf-to-dsphid}
        |d^S(\bar{f}(x))-d^S(\Phi_{\bar f}(x,r))|\leq \varepsilon r \quad \text{whenever $d^{\partial\Sigma}_0(x)<2\alpha_0$ and $|r|<\bar r'$},
    \end{equation}
    using \eqref{eq:choice-of-rD'}. 
    Moreover, combining \eqref{eq:choice-of-rD'} and \eqref{eq:ftc}, 
    \begin{align}
        \|\partial_r \big(d^S\circ \bar{f} -d^S\circ \Phi_{\bar f}\big)\|_{C^0(B_{2\alpha_0}(\partial\Sigma)\times(-\bar r',\bar r'))} < \varepsilon.
    \end{align}
    Letting
    \begin{equation}
        C''\vcentcolon=\|\nabla_{g_0}\langle N^S\circ \Pi^S\circ \Phi_{\bar f},\nu_{\bar{f}}\rangle \|_{C^0(B_{2\alpha_0}(\partial\Sigma)\times(-\bar r',\bar r'))} < \infty
    \end{equation}
    and choosing $\bar r''\in(0,\min\{\varepsilon/C'',\bar r'\})$, \eqref{eq:ftc} yields
    \begin{equation}
        \|\nabla_{g_0} \big(d^S\circ \bar{f} -d^S\circ \Phi_{\bar f}\big)\|_{C^0(B_{2\alpha_0}(\partial\Sigma)\times(-\bar r'',\bar r''))} < \varepsilon.
    \end{equation}
    Together, \eqref{eq:C1-perturb-est-Phi} follows. Using \eqref{eq:def-zeta}
    and \Cref{rem:dpartialsigma}, we estimate
    \begin{equation}
        \|\zeta\circ d^{\partial\Sigma}_0\|_{C^1(\Sigma)}\leq \frac{C(g_0)}{\alpha_0}.
    \end{equation} 
    Since furthermore $\|N^S\circ\Pi^S\circ\Phi_{\bar f}\|_{C^1(\Phi_{\bar f}^{-1}(\mathcal{U}_S))}<\infty$, there exists $\bar r\in (0,\bar r'')$, depending on $\varepsilon$ and $\alpha_0$, so that the claim follows with \eqref{eq:dec-of-Phi}, \eqref{eq:C1-perturb-est-Phi} and \eqref{eq:est-dsbarf-to-dsphid}.
\end{proof}

\begin{proof}[Proof of \Cref{prop:ex-gauss-coord}.]

Since the set of diffeomorphisms is open in the $C^1$-topology, \Cref{lem:C1-perturb-est-Phi} (with $\varepsilon>0$ sufficiently small) yields that each $\Phi|_{\Sigma_j\times(-\bar r,\bar r)}$ is a diffeomorphism onto its image, where $\Phi$ is defined as in \eqref{eq:def-Phi}.

Statement (i) is immediate from the definition.

With \eqref{eq:def-Phi-D}, \eqref{eq:dec-of-Phi} and (i), for all $x\in \Sigma$
\begin{align}
\xi(x,0)=\left.\partial_r \Phi(x,r)\right|_{r=0} & =  \nu_{\bar{f}}(x)- \zeta(d^{\partial\Sigma}_0(x)) \left. \partial_r\right|_{r=0}\Big[ d^S(\Phi_{\bar f}(x,r))\Big]N^S(\Pi^S(\bar{f}(x))) \\
& =  \nu_{\bar{f}}(x) - \zeta(d^{\partial\Sigma}_0(x)) \langle N^S(\Pi^S(\bar{f}(x))), \nu_{\bar{f}}(x) \rangle N^S(\Pi^S(\bar{f}(x)))\label{eq:xi-circ-barf}
\end{align}
since,  by \eqref{eq:abldistance}, $\partial_r \Big[d^S(\Phi_{\bar f}(x,r))\Big] = \langle N^S(\Pi^S(\Phi_{\bar f}(x,r))), \nu_{\bar{f}}(x) \rangle $ for $x\in B_{2\alpha_0}(\partial\Sigma)$ and  $|r|< \tilde{r}_{\bar f}(\alpha_0)$. In particular, $\xi(x,0)=\nu_{\bar{f}}(x)$ on $\partial \Sigma$ follows using \eqref{intro-eq:fbc}. Furthermore, since $0<\varepsilon<\frac12$ in \Cref{lem:C1-perturb-est-Phi}, by \eqref{eq:choice-of-rD'}, we have $\langle\xi(\cdot,0),\nu_{\bar{f}}\rangle \in (\frac12,2)$ on $\Sigma$, that is \eqref{eq:controlxi}.

Next, we verify property (iii). Using \eqref{eq:Phi-bdry-values}, for all $x\in B_{\alpha_0}(\partial\Sigma)$ and $|r|<\bar r$ we have 
$$\mathrm{dist}(\Phi(x,r),S)=\mathrm{dist}(\Phi_S(x,r),S)\leq d^S(\bar f(x))$$ so that \eqref{eq:etildaS} yields $\Phi(x,r)\in \mathcal{U}_S$. In particular, $\Phi(x,r)\in S$ for $x\in\partial\Sigma$ since $d^S(\bar{f}(x))=0$ by \eqref{intro-eq:fbc}.

It remains to verify (iv). 
The map $\Phi_{\bar f}(\psi(x),r)$ is linear in $r$ and hence real-analytic in $r$. The uniformity in $x$ follows by smoothness of $\Phi_{\bar f}$ and considering the map on 
$\bar\Omega \times [-\bar r,\bar r]$. By \Cref{rem:geomS}, since the composition of analytic maps is again analytic\footnote{see \cite[Theorem~2.2.8]{KrantzParks}. Here we use also \cite[Proposition~2.2.10]{KrantzParks} and Remark \ref{rem:B2}.}, we see that also $\Phi_S(\psi(x),r)$ is analytic in $r$, uniformly in $x$ by smoothness of the mappings involved and considering an appropriate compact set. Since the map $\zeta$ is smooth, the mapping $\Phi(\psi(x),r)$ is analytic in $r$, uniformly in $x$. One argues similarly for $\partial_x^{\alpha}\Phi(\psi(x),r)$ and $\partial_x^{\alpha}\partial_r\Phi(\psi(x),r)$. Altogether, (iv) is verified.
\end{proof}

\begin{lemma}\label{lem:def-fw}
    For
    \begin{equation}
        U_0\vcentcolon=\{w\in W^{4,2}(\Sigma,\R)\mid \|w\|_{C^0(\Sigma)}<\bar r\},
    \end{equation}
    the mapping $U_0\to W^{4,2}_{\mathrm{imm}}(\Sigma,\R^3)$, $w\mapsto f_w\vcentcolon= \Phi(\cdot,w(\cdot))$ is analytic and
    \begin{equation}\label{eq:0th-order-fbc}
        f_w(\partial \Sigma)\subset S\quad\text{for all $w\in U_0$}.
    \end{equation}
    Moreover, there exists $a\in (0,\bar r)$ such that, for $U_1\vcentcolon=\{w\in U_0:\|w\|_{C^1(\Sigma)}<a\}$,
    \begin{equation}\label{eq:estimate-xi0fw-in-nu-fw}
        \frac14\leq \langle\xi(\cdot,w(\cdot)),\nu_{f_w}\rangle\leq 4\quad\text{on $\Sigma$}\quad\text{for all $w\in U_1$}.
    \end{equation}
\end{lemma}
Here and in the following, the Sobolev and Lebesgue spaces are understood with respect to the reference metric $g_0$ in \Cref{rem:dpartialsigma}.

\begin{proof}[Proof of \Cref{lem:def-fw}]
    The analyticity of $U_0\to W^{4,2}(\Sigma)$, $w\mapsto f_w$ follows from  \Cref{thm:analytic_composition,lem:analit-local} (with $m=n=4, p=2$) using \Cref{prop:ex-gauss-coord}\eqref{item:ex-gauss-coord_3}. For $x\in \Sigma$, we have in local coordinates $x=(x_1,x_2)$
    \begin{equation}
        Df_w(x) = D\Phi(x,w(x)) \cdot \begin{pmatrix}
            \mathrm{Id}_{2\times 2}\\ Dw(x) 
        \end{pmatrix}.
    \end{equation}
    Since, locally, $\Phi$ is a diffeomorphism, $D\Phi(x,w(x))$ is invertible so that 

        $\mathrm{rank}\ Df_w(x) 
        = 2.$
    Therefore, $f_w$ is again an immersion for each $w\in U_0$. Equation~\eqref{eq:0th-order-fbc} is a consequence of \Cref{prop:ex-gauss-coord}\eqref{item:PhivaluesS}. 
    
    Clearly, $C^1(\Sigma)\supseteq\{w\in W^{4,2}(\Sigma):\|w\|_{C^0(\Sigma)}<\bar r\}\ni w\mapsto \langle \xi(\cdot,w(\cdot)),\nu_{f_w}\rangle\in C^0(\Sigma)$ is continuous. Thus \eqref{eq:estimate-xi0fw-in-nu-fw} follows from \eqref{eq:controlxi} 
    choosing $a\in (0,\bar r)$ sufficiently small.
\end{proof}

\subsection{Rewriting the free boundary problem in Gaussian coordinates}

As in the previous section, we fix a smooth immersion $\bar{f}\colon\Sigma\to\R^3$ of $\Sigma$ satisfying \eqref{intro-eq:fbc}. Choose 
\begin{equation}\label{eq:spaces}
    V\vcentcolon=W^{4,2}(\Sigma), \quad Z\vcentcolon=L^2(\Sigma) \quad\text{ and }\quad Y\vcentcolon= W^{\frac52,2}(\partial \Sigma)\times W^{\frac12,2}(\partial \Sigma) .
\end{equation}
For $w \in U_1\subset V$, neighborhood of $0$ as in \Cref{lem:def-fw}, we write $f_w=\Phi(\cdot, w(\cdot))$ with $\Phi$ as in \Cref{prop:ex-gauss-coord} and recall that $f_{w=0}=\bar{f}$. 
Now, define 
\begin{equation}\label{eq:def-E}
    \E\colon U_1\to\R, \quad \E(w)\vcentcolon= \W(f_w)
\end{equation}
and 
\begin{equation}\label{eq:gradientWillmore}
    \delta\E\colon U_1\to Z, \quad  \delta\E(w)\vcentcolon=\frac{1}{\langle \xi(\cdot,w(\cdot)),\nu_{f_w}\rangle}\Big[\Delta_{g_{f_w}} H_{f_w} + |A^0_{f_w}|^2H_{f_w}\Big].
\end{equation}
The factor in front of the usual Willmore gradient, which is uniformly bounded by \eqref{eq:estimate-xi0fw-in-nu-fw}, is introduced here to simplify some calculations. 
For the {\L}ojasiewicz--Simon inequality \Cref{intro-thm:loja-abstract}, real analyticity is a key property. Since multilinear and bounded maps are analytic, the same is true for many geometric quantities.

\begin{lemma}\label{lem:analytic}
The following maps are analytic.
\begin{enumerate}[(i)]
\item\label{item:analytic_2} $U_1\to W^{3,2}(\Sigma;\R^3), w\mapsto \nu_{f_w}$;
\item\label{item:analytic_6} $U_1\to W^{3,2}(\Sigma;\R), w\mapsto (\langle \xi(\cdot,w(\cdot)),\nu_{f_w}\rangle)^{-1}$;
\item\label{item:analytic_3} $U_1\to W^{2,2}(\Sigma), w\mapsto H_{f_w}$;
\item\label{item:analytic_4} $\delta \mathcal{E}\colon U_1\to L^{2}(\Sigma)$;
\item \label{item:analytic_5} $\mathcal{E}\colon U_1\to\R$.
\end{enumerate}
\end{lemma}
\begin{proof} 
By \Cref{lem:def-fw}, $U_1\to W^{4,2}_{\mathrm{imm}}(\Sigma;\R^3), w\mapsto f_w$ is analytic. 
For \eqref{item:analytic_2}, we consider local positive coordinates $x=(x^1,x^2)\colon U\overset{\approx}{\to}\Omega$ with $U\subset \Sigma$ and $\Omega\subset\R\times[0,\infty)$ (relatively) open. Since it is linear and bounded, 
$W^{4,2}(\Omega;\R^3)\to W^{3,2}(\Omega;\R^3), f\mapsto \partial_i f$ is analytic for $i=1,2$, and thus so is $W^{4,2}(\Omega;\R^3)\to W^{3,2}(\Omega;\R^3)$, $f\mapsto \partial_1f \times \partial_2 f$. The map $\R^3\setminus \{0\} \to \R^3, y\mapsto \frac{y}{|y|}$ is analytic. By \Cref{lem:analit-local} and the local representation of the normal, cf.\ \eqref{eq:def_normal}, it follows that $U_1\to W^{3,2}(\Sigma;\R^3), w\mapsto \nu_{f_w}$ is analytic. 

Moreover, by the arguments in \cite[Lemma 3.2]{chillfasangovaschaetzle2009}, the following maps are analytic.
    \begin{align}  
    &W^{4,2}_{\mathrm{imm}}(\Sigma;\R^3) \to C^0(\Sigma), f\mapsto \frac{\dd\mu_f}{\dd \mu_{g_0}}\text{ with $\mu_{g_0}$ as in \Cref{rem:dpartialsigma}} \\
    &W^{4,2}_{\mathrm{imm}}(\Sigma;\R^3) \to W^{2,2}(\Sigma;\R^3), f\mapsto H_f \nu_f\\
        &W^{4,2}_{\mathrm{imm}}(\Sigma;\R^3) \to L^{2}(\Sigma;\R^3), f\mapsto (\Delta_{g_f} H_f + |A^0_f|^2 H_f)\nu_f. \label{eq:analytic_CFS}
    \end{align}
    Since they are multilinear and bounded, the multiplications
    \begin{align}
        &W^{4,2}(\Sigma;\R^3)\times  W^{3,2}(\Sigma;\R^3) \to W^{3,2}(\Sigma), (h_1,h_2)\mapsto \langle h_1,h_2\rangle\\
        &W^{2,2}(\Sigma;\R^3)\times  W^{3,2}(\Sigma;\R^3) \to W^{2,2}(\Sigma), (h_1,h_2)\mapsto \langle h_1,h_2\rangle\\
        & L^2(\Sigma;\R^3)\times L^\infty(\Sigma;\R^3) \to L^2(\Sigma), (h_1,h_2)\mapsto \langle h_1,h_2\rangle
    \end{align}
    are analytic. Using \Cref{prop:ex-gauss-coord}\eqref{item:ex-gauss-coord_3} and \Cref{thm:analytic_composition,lem:analit-local} (with $m=n=4$, $p=2$), the mapping $U_1\to W^{4,2}(\Sigma,\R^3)$, $w\mapsto (\partial_r\Phi)(\cdot,w(\cdot)) = \xi(\cdot,w(\cdot))$ is analytic. Combined with part \eqref{item:analytic_2}, also $U_1\to W^{3,2}(\Sigma)$, $w\mapsto \langle \xi(\cdot,w(\cdot)),\nu_{f_w}\rangle$ is analytic. Finally, due to \eqref{eq:estimate-xi0fw-in-nu-fw}, part \eqref{item:analytic_6} follows, since $(0,\infty)\ni t\mapsto t^{-1}$ is analytic.
    
    Parts \eqref{item:analytic_3} and \eqref{item:analytic_4} thus follow from \eqref{item:analytic_2} and \eqref{eq:analytic_CFS}, recalling \eqref{eq:gradientWillmore} and part \eqref{item:analytic_6}. Part \eqref{item:analytic_5} follows from \eqref{item:analytic_3} and \eqref{eq:analytic_CFS}.
\end{proof}

\begin{lemma}\label{lem:reformulation-first-order-bdry-condition}
    Denote by $\vec{\mu}=(\vec\mu^{(1)},\vec\mu^{(2)})\colon \partial \Sigma\times(-\bar r,\bar r)\to T\Sigma\times\R$ the smooth mapping with
    \begin{equation}\label{eq:vectormu}
        \vec{\mu}(x,r)\vcentcolon=(d\Phi_{(x,r)})^{-1}.N^S(\Phi(x,r))
    \end{equation}
    the representation of $N^S$ in generalized Gaussian coordinates $(x,r)$. 
    
    Then $f=f_w$ with $w\in W^{4,2}(\Sigma,(-\bar r,\bar r))$ satisfies the second condition in \eqref{intro-eq:fbc} if and only if    \begin{equation}
        \partial_{\vec\mu^{(1)}(x,w(x))}w(x)-\vec\mu^{(2)}(x,w(x))=0\quad\text{for all $x\in\partial\Sigma$}.
    \end{equation}
\end{lemma}
\begin{proof}   
    For $f=f_w$, we compute
    \begin{align}
        \langle &\nu_{f_w}(x),N^S(f_w(x))\rangle =\langle \nu_{f_w}(x),d\Phi_{(x,w(x))}. \vec\mu(x,w(x))\rangle\\
        &= g_0\Big((d\Phi_{(x,w(x))}^*\big)^{(1)}.\nu_{f_w}(x),\vec\mu^{(1)}(x,w(x))\Big) + (d\Phi_{(x,w(x))}^*\big)^{(2)}.\nu_{f_w}(x) \cdot \vec\mu^{(2)}(x,w(x))
    \end{align}
    for all $x\in\partial\Sigma$. Here, $d\Phi_{(x,w(x))}^*$ is the adjoint of
    \begin{align}\label{eq:adjoint}
        d\Phi_{(x,w(x))} \colon (T_x\Sigma\times \R, g_0(x)\otimes \dd r)\to(\R^3, \langle\cdot,\cdot\rangle).
    \end{align}    
    In particular, if, for some $\lambda\in\R$,
    \begin{equation}\label{eq:reformulation-first-order-bdry-condition-1}
        d\Phi_{(x,w(x))}^*.\nu_{f_w}(x) = \lambda (\nabla_{g_0}w(x),-1),
    \end{equation}
    then $\lambda\neq 0$ and the claim follows.
    To verify \eqref{eq:reformulation-first-order-bdry-condition-1}, we compute for $x\in \Sigma$ and a basis ${e_1,e_2}$ of $T_x\Sigma$
    \begin{align}
        &\langle \partial_{e_i}f_w(x),\big(d\Phi_{(x,w(x))}^*\big)^{-1}. (\nabla_{g_0} w(x),-1)\rangle \\
        &= \langle d\Phi_{(x,w(x))}. \begin{pmatrix} e_i\\ \partial_{e_i}w(x) \end{pmatrix},\big(d\Phi^*_{(x,w(x))}\big)^{-1}. (\nabla_{g_0} w(x),-1)\rangle=\partial_{e_i}w(x)-\partial_{e_i}w(x)=0.&& \qedhere
    \end{align}
\end{proof}

\begin{lemma}\label{lem:choice-of-bdryconditions}
With $\mathcal{B}\colon U_1\to Y$, 
\begin{align}\label{eq:boundcondi}
    w\mapsto \mathcal{B}(w) \vcentcolon= \big(&\partial_{\vec\mu^{(1)}(x,w(x))}w(x)-\vec\mu^{(2)}(x,w(x)) \Big|_{\partial\Sigma}, \\
    &\frac{1}{\langle \xi(\cdot,w(\cdot)),\nu_{f_w}\rangle}\Big[\frac{\partial H_{f_w}}{\partial \eta_{f_w}} - (\nabla^2 d^S)\circ f_w~(\nu_{f_w},\nu_{f_w})H_{f_w}\Big]\Big|_{\partial \Sigma}\big),
\end{align}
we have that $\mathcal{M}\vcentcolon= \B^{-1}(\{(0,0)\})$ is the set of $w\in U_1$ such that $f_w$ satisfies the free boundary conditions \eqref{intro-eq:fbc}. Moreover, $\mathcal{B}$ is analytic.
\end{lemma}
\begin{proof}
The first part of the claim follows from \eqref{eq:abldistance}, \Cref{lem:reformulation-first-order-bdry-condition}, and 
\eqref{eq:estimate-xi0fw-in-nu-fw}.  
    Next, we verify the analyticity of $\mathcal B$. 
    
    Let $\Omega\subset \R\times[0,\infty)$ be open and let $\psi\colon \Omega\overset{\approx}{\to}\psi(\Omega)\subset \Sigma$ be a local inverse chart with bounded derivatives as in \Cref{lem:analit-local}. Define $\tilde\Phi\colon\bar\Omega\times(-\bar r,\bar r)\to\R^3$, $\tilde\Phi(y,r)\vcentcolon=\Phi(\psi(y),r)$. Define
    $$
        \tilde\mu(y,r)\vcentcolon=(D\tilde\Phi(y,r))^{-1}N^S(\tilde\Phi(y,r)) =\vcentcolon (\tilde\mu^{(1)}(y,r),\tilde\mu^{(2)}(y,r))\in\R^2\times\R.
    $$
    By Cramer's rule, $\{D\in \R^{2\times 2}\mid\det D\neq 0\}\to \R^{2\times 2}$, $D\mapsto D^{-1}$ is analytic. Therefore, using \Cref{prop:ex-gauss-coord}\eqref{item:ex-gauss-coord_3} and the fact that $S$ is analytic, $\psi^{-1}(B_{\alpha_0}(\partial\Sigma)) \times (-\bar r,\bar r)\to \R^3$, $(y,r)\mapsto \partial_y^\alpha\tilde\mu(y,r)=\partial_y^\alpha[(D\tilde\Phi(y,r))^{-1}(\nabla d^S)(\tilde\Phi(y,r))]$ is analytic in $r$, uniformly in $y$, for all $\alpha\in\N_0^2$. 

    Identifying $T_x\Sigma$ with a subset of $\R^M$ for each $x\in\Sigma$ using Nash's embedding theorem, \Cref{thm:analytic_composition,lem:analit-local} (with $m=n=4$, $p=2$) yield that 
    $$
        U_1\to W^{4,2}(B_{\alpha_0}(\partial\Sigma),\R^{M+1}),\ w\mapsto {\vec\mu(\cdot,w(\cdot))}\big|_{B_{\alpha_0}(\partial\Sigma)}
    $$
    is analytic.
    Thus, it follows that 
    \begin{equation}
        U_1\to W^{3,2}(B_{\alpha_0}(\partial\Sigma),\R),\ w\mapsto \partial_{\vec\mu^{(1)}(\cdot,w(\cdot))}w(\cdot) -\vec\mu^{(2)}(\cdot,w(\cdot)) \big|_{B_{\alpha_0}(\partial\Sigma)}
    \end{equation}
    is analytic. By composing with the analytic trace operator, we find that $U_1\to W^{\frac52,2}(\partial \Sigma,\R)$, $w\mapsto \partial_{\vec\mu^{(1)}(\cdot,w(\cdot))}w(\cdot)-\vec \mu^{(2)}(\cdot,w(\cdot))|_{\partial \Sigma}$ is analytic. Altogether, analyticity of the first component is established.

    For the second component, using \Cref{lem:analytic}\eqref{item:analytic_6} and analyticity of the linear trace operator, we only need to show analyticity in $w$ of the term in brackets.
    
    Let $R$ be the open set of $C^1$ Riemannian metrics on $\Sigma$. 
    By \eqref{eq:eta_g-expl-form} we have that, $R\to C^1(\partial\Sigma,\R^M)$, $g\mapsto \eta_g$ is analytic.
    Using that $\mathrm tr_{\partial\Sigma}\colon W^{1,2}(\Sigma)\to W^{1/2,2}(\partial\Sigma)$,
    \begin{align}
        R\times W^{2,2}(\Sigma)\to W^{1/2,2}(\partial \Sigma), (g,h)\mapsto \partial_{\eta_g} h = d h.\eta_g
    \end{align}
    is well-defined, bilinear, bounded, and thus analytic. In particular, as $U_1\to R$, $w\mapsto g_{f_w}$ is analytic, also 
    $$U_1\to W^{1/2,2}(\partial\Sigma),\ w\mapsto \frac{\partial H_{f_w}}{\partial \eta_{f_w}}=dH_{f_w}.\eta_{f_w},$$ is analytic.
    On the other hand, by \eqref{eq:sign_dist}, we have that $\nabla^2 d^S$ is analytic in $\mathcal{U}_S$ with values in the bilinear maps on $\R^3$. 
    Also using \Cref{prop:ex-gauss-coord}\eqref{item:PhivaluesS}, \Cref{lem:analytic}\eqref{item:analytic_2} and \Cref{lem:def-fw}, it follows that
    \begin{align}
        U_1 \to W^{3,2}(B_{\alpha_0}(\partial\Sigma)), w\mapsto (\nabla^2 d^S)\circ f_w~(\nu_{f_w},\nu_{f_w})\Big|_{B_{\alpha_0}(\partial\Sigma)}
    \end{align}
    is analytic. Now, \Cref{lem:analytic}\eqref{item:analytic_3} and analyticity of the trace imply that the term in brackets in the second component of $\mathcal{B}$ is analytic. As noted above, this is sufficient to conclude the analyticity of $\mathcal B$.
\end{proof}

\begin{proposition}\label{lem:first-variations-B-and-E}
    For $\varphi\in W^{4,2}(\Sigma)$, we have that
    \begin{align}
        \mathcal{B}'(0)\varphi = \big(& \frac{\partial}{\partial\eta_{\bar f}} \varphi + A^S(\nu_{\bar f},\nu_{\bar f})\varphi \Big|_{\partial\Sigma},  \\
        &\frac{\partial}{\partial\eta_{\bar f}} \Delta_{g_{\bar f}} \varphi + A^S(\nu_{\bar f},\nu_{\bar f})\Delta_{g_{\bar f}}\varphi +  d\varphi.\vec b_{1} + b_0\cdot \varphi \Big|_{\partial\Sigma}\big) \label{eq:deltabprime}  
    \end{align}
    for suitable coefficients $b_0\colon\partial\Sigma\to\R$ and $\vec b_{1}\colon\partial\Sigma\to T\Sigma$ depending on $\bar{f}$. Furthermore,
    \begin{equation}\label{eq:deltaeprime}
        (\delta\E)'(0) \varphi - (\Delta_{g_{\bar f}})^2 \varphi \in W^{1,2}(\Sigma) 
        \quad\text{for all $\varphi\in W^{4,2}(\Sigma)$}.
    \end{equation}
\end{proposition}

For the proof of \Cref{lem:first-variations-B-and-E}, we first collect the following fundamental computations.

\begin{lemma}\label{lem:loc-coord-w-loja-sec}
    Let $w\in W^{4,2}(\Sigma)$ with $\|w\|_{C^0}<\bar r$.
        If $\partial_{\eta_{f_w}}f_w = N^S\circ f_w$ on $\partial \Sigma$, then $$\vec \mu(x,w(x))=(\eta_{f_w}(x),\partial_{\eta_{f_w}}w(x))$$ for all $x\in\partial\Sigma$. In particular, $\vec\mu(\cdot,0)=(\eta_{\bar f},0)$.
    Moreover, for $\varphi\in W^{4,2}(\Sigma)$, we have the following variation formulas.
    \begin{align}
        &\partial_{t,0} H_{f_{t\varphi}} = \Delta_{g_{\bar f}} (\varphi\langle \xi(\cdot,0),\nu_{\bar f}\rangle)+|A_{\bar f}|^2\varphi\langle\xi(\cdot,0),\nu_{\bar f}\rangle \label{eq:lin-iden-H}\\
        &\partial_{t,0} \nu_{f_{t\varphi}} = -\mathrm{grad}_{g_{\bar f}} (\varphi\langle\xi(\cdot,0),\nu_{\bar f}\rangle)\label{eq:lin-iden-nu}\\
        &\partial_{t,0}\eta_{f_{t\varphi}} = \varphi\cdot (c_1(A_{\bar f},g_{\bar f})\tau_0+c_2(A_{\bar f},g_{\bar f})\eta_0).\label{eq:lin-iden-eta}
    \end{align}
    where $c_1$, $c_2$ are suitable smooth functions. Finally, for any $h\colon\Sigma \to\R$
    \begin{align}
        (\partial_{t,0} \Delta_{g_{f_{t\varphi}}}) h =\ &2\langle A_{\bar f},\nabla^2h\rangle_{g_{\bar f}}\varphi\langle\xi(\cdot,0),\nu_{\bar f}\rangle - 2\langle \nabla^* (\varphi \langle\xi(\cdot,0),\nu_{\bar f}\rangle A_{\bar f}) \\
        &- \frac12\nabla(\varphi\langle\xi(\cdot,0),\nu_{\bar f}\rangle H_{\bar f}),\nabla h\rangle_{g_{\bar f}}.\label{eq:lin-iden-LapBel}
    \end{align}
\end{lemma}
\begin{proof}
    By the definition of $\vec{\mu}$, we compute for $x\in\partial \Sigma$
    \begin{align}
        d\Phi_{(x,w(x))}.\vec{\mu}(x,w(x)) &= N^S(f_w(x)) =  \partial_{\eta_{f_w}}f_w(x) =  \partial_{\eta_{f_w}} (\Phi(x,w(x)))\\
        &= d\Phi_{(x,w(x))} . \begin{pmatrix} \eta_{f_w}(x)\\\partial_{\eta_{f_w}}w(x) \end{pmatrix}
    \end{align}
    so that
    the claimed formula for $\vec \mu$ follows. We next verify the variation identities. We obtain the normal speed $\langle\partial_{t,0} f_{t\varphi},\nu_{\bar f}\rangle = \varphi\langle \xi(\cdot,0),\nu_{\bar f}\rangle$. Therefore, the variation formulas for $H$ and $\nu$ follow, for instance using \cite[Lemma~2.3]{rupp2023}. The same reference yields $\partial_{t,0} g = -2A_{\bar f}\varphi\langle\xi(\cdot,0),\nu_{\bar f}\rangle$. Thus, the variation identity for $\eta$ follows from \eqref{eq:eta_g-expl-form} and the variation formula for the Laplace--Beltrami operator is due to \cite[Proposition~2.3.10]{topping2006}.
\end{proof}

\begin{proof}[Proof of \Cref{lem:first-variations-B-and-E}]    
    First, we verify the formula for $(\mathcal{B}^{(1)})'(0)$. 
    Using \Cref{lem:loc-coord-w-loja-sec},
    \begin{align}
        \partial_{t,0} \big(\partial_{\vec\mu^{(1)}(\cdot,t\varphi)}t\varphi - \vec\mu^{(2)}(\cdot,t\varphi)\big) = \partial_{\eta_{\bar f}} \varphi - \partial_{t,0}\vec\mu^{(2)}(\cdot,t\varphi).    \label{eq:linear-B1-2}
    \end{align}
    Notice that by \eqref{eq:vectormu}, \eqref{eq:adjoint}
    \begin{align}
        - \partial_{t,0}\vec\mu^{(2)}(\cdot,t\varphi) &= \langle (d\Phi^*_{(\cdot,0)})^{-1}.(0,-1),\partial_{t,0}N^S(\Phi(\cdot,t\varphi(\cdot)))\rangle  \\
        &\quad + \langle \partial_{t,0} (d\Phi^*_{(\cdot,t\varphi(\cdot))})^{-1}.(0,-1) , N^S(\bar f)\rangle\\
        &= \langle (d\Phi^*_{(\cdot,0)})^{-1}.(0,-1),\partial_{t,0}N^S(\Phi(\cdot,t\varphi(\cdot)))\rangle  \\
        &\quad - \langle  (d\Phi^*_{(\cdot,0)})^{-1}.\partial_{t,0}(d\Phi^*_{(\cdot,t\varphi(\cdot))}).(d\Phi^*_{(\cdot,0)})^{-1}.(0,-1) , N^S(\bar f)\rangle.\label{eq:linear-B1-1}
    \end{align}
    For any $w=d\Phi_{(x,0)}.(v,\rho)\in\R^3$, we compute for $x \in \partial \Sigma$
    $$\langle (d\Phi^*_{(\cdot,0)})^{-1}.(0,-1),w\rangle = (g_0(x)\otimes \dd r)\big((0,-1),(d\Phi_{(x,0)})^{-1}.w\big)=-\rho = -\langle\nu_{\bar f}, w\rangle,$$ so
    \begin{equation}
        (d\Phi^*_{(\cdot,0)})^{-1}.(0,-1)=-\nu_{\bar f} \quad\text{and}\quad (d\Phi_{(\cdot,0)})^{-1}N^S(\bar f)=\vec\mu(\cdot,0)=(\eta_{\bar f},0),
    \end{equation}
    using \Cref{lem:loc-coord-w-loja-sec} 
     in the second equation. Thus, combined with $\partial_r\Phi=\xi$, we obtain
    \begin{align}
        \langle  &(d\Phi^*_{(\cdot,0)})^{-1}.\partial_{t,0}(d\Phi^*_{(\cdot,t\varphi(\cdot))}).(d\Phi^*_{(\cdot,0)})^{-1}.(0,-1) , N^S(\bar f)\rangle \\
        &= \langle (d\Phi^*_{(\cdot,0)})^{-1}.(0,-1) , \partial_{t,0}(d\Phi_{(\cdot,t\varphi(\cdot))}).(d\Phi_{(\cdot,0)})^{-1}.N^S(\bar f)\rangle \\
        &= -\langle \nu_{\bar f} , \partial_{t,0}(d\Phi_{(\cdot,t\varphi(\cdot))}).(\eta_{\bar f},0)\rangle \\
        &= - \langle \nu_{\bar f}, \partial_{t,0}(\partial_{\eta_{\bar f}}\Phi(\cdot,t\varphi))\rangle = - \langle\nu_{\bar f}, \partial_{\eta_{\bar f}}(\xi(\cdot,0)) \varphi\rangle.
    \end{align}
    Plugging the above into \eqref{eq:linear-B1-1}, and using $N^S=\nabla d^S$ on $S$ and $\partial_{r,0}\Phi=\nu_{\bar f}$ on $\partial\Sigma$, yields
    \begin{align}
        - \partial_{t,0}&\vec\mu^{(2)}(\cdot,t\varphi) = -\nabla^2d^S\circ \bar f(\nu_{\bar f},\nu_{\bar f})\varphi - \langle\nu_{\bar f},\partial_{\eta_{\bar f}}(\xi(\cdot,0))\rangle \varphi. 
    \end{align}
    Finally, using \eqref{eq:xi-circ-barf} we have 
    $\langle\nu_{\bar f},\partial_{\eta_{\bar f}}\xi(\cdot,0)\rangle=0$ on $\partial \Sigma$. 
    Thus, by \eqref{eq:linear-B1-2},
    \begin{align}
        \partial_{t,0} \big(\partial_{\vec\mu^{(1)}(\cdot,t\varphi)}t\varphi - \vec\mu^{(2)}(\cdot,t\varphi)\big) = \partial_{\eta_{\bar f}} \varphi -\nabla^2d^S\circ \bar f(\nu_{\bar f},\nu_{\bar f})\varphi. 
    \end{align}

    Next, using \eqref{eq:lin-iden-H}, \eqref{eq:lin-iden-eta} and again $\xi(\cdot,0)=\nu_{\bar f}$ on $\partial\Sigma$, we have
    \begin{align}
        \partial_{t,0} &\Big( \frac{\partial_{\eta_{f_{t\varphi}}}H_{f_{t\varphi}}-(\nabla^2d^S)\circ f_{t\varphi}~(\nu_{f_{t\varphi}},\nu_{f_{t\varphi}})H_{f_{t\varphi}}}{\langle\xi(\cdot,t\varphi(\cdot)),\nu_{f_{t\varphi}}\rangle} \Big) \\
        &= \partial_{\eta_{\bar f}} (\partial_{t,0}H_{f_{t\varphi}}) - (\nabla^2d^S) \circ\bar f~(\nu_{\bar f},\nu_{\bar f})\partial_{t,0}H_{f_{t\varphi}} + d\varphi.\vec{\tilde b}_1 + \varphi\cdot \tilde b_0 \\
        &= \partial_{\eta_{\bar f}} \Delta_{g_{\bar f}}\varphi 
        + A^S~(\nu_{\bar f},\nu_{\bar f})\Delta_{g_{\bar f}}\varphi + d\varphi.\vec b_1 + \varphi\cdot b_0,
    \end{align}
    using \eqref{eq:abldistance},
    for suitable smooth coefficients $\vec{\tilde b}_1$, $\vec b_1$ and $\tilde b_0$, $b_0$ depending on $\bar f$. Finally, to verify \eqref{eq:deltaeprime}, we first observe that
    \begin{equation}
        \partial_{t,0} (\Delta_{g_{f_{t\varphi}}} H_{f_{t\varphi}}) = (\partial_{t,0} \Delta_{g_{f_{t\varphi}}}) H_{\bar f} + \Delta_{g_{\bar f}} (\partial_{t,0}H_{f_{t\varphi}}).
    \end{equation}
    Together with \eqref{eq:lin-iden-H} and \eqref{eq:lin-iden-LapBel}, we get that 
    \begin{equation}
        \frac{\partial_{t,0} (\Delta_{g_{f_{t\varphi}}} H_{f_{t\varphi}})}{\langle \xi(\cdot,0),\nu_{\bar f}\rangle} - \Delta_{g_{\bar f}}^2 \varphi \in W^{1,2}(\Sigma).
    \end{equation}
    Also using \eqref{eq:lin-iden-nu}, we have
    \begin{equation}
        \delta\E'(0).\varphi - \frac{\partial_{t,0} (\Delta_{g_{f_{t\varphi}}} H_{f_{t\varphi}})}{\langle \xi(\cdot,0),\nu_{\bar f}\rangle} = \partial_{t,0}\Big(\frac{\Delta_{g_{f_{t\varphi}}}H_{f_{t\varphi}} + |A_{f_{t\varphi}}^0|^2H_{f_{t\varphi}}}{\langle \xi(\cdot,t\varphi(\cdot)),\nu_{f_{t\varphi}}\rangle}\Big) - \frac{\partial_{t,0} (\Delta_{g_{f_{t\varphi}}} H_{f_{t\varphi}})}{\langle \xi(\cdot,0),\nu_{\bar f}\rangle} \in W^{1,2}(\Sigma),
    \end{equation}
    noting that
    \begin{align}
    \partial_{t,0}\Big(\frac{|A_{f_{t\varphi}}^0|^2H_{f_{t\varphi}}}{\langle \xi(\cdot,t\varphi(\cdot)),\nu_{f_{t\varphi}}\rangle}\Big)\in W^{2,2}(\Sigma).
    \end{align}
    In particular, \eqref{eq:deltaeprime} follows.
\end{proof}

\subsection{The \L ojasiewicz--Simon inequality for free boundary surfaces in Gaussian coordinates}

The following result is standard in literature. For instance, it can be obtained as a corollary of \cite[Chapter 5, Propositions~7.6 and 7.7]{taylor2023}.

\begin{lemma}\label{thm:neumann-bilaplace}
    For any Riemannian metric $g$ on $\Sigma$ with outer unit conormal $\eta$ as in \eqref{eq:eta_g-expl-form}, the operator
    \begin{align}
        &L\colon W^{4,2}(\Sigma)\to L^2(\Sigma)\times W^{\frac52,2}(\partial\Sigma)\times W^{\frac12,2}(\partial\Sigma),\\
        &u\mapsto Lu \vcentcolon= \big(\Delta_g^2 u, \frac{\partial u}{\partial\eta},\frac{\partial\Delta_gu}{\partial\eta}\big)
    \end{align}
    is a Fredholm operator of index $0$.
\end{lemma}

\begin{proposition}\label{prop:PDE-argument}
    Let $\Sigma$ be equipped with a Riemannian metric $g$, consider a smooth function $m\colon\partial \Sigma\to\R$ as well as smooth coefficients $\vec b\colon\partial\Sigma\to T\Sigma$, $b\colon\partial\Sigma\to\R$. Then there exist constants $c_1,c_2>0$ such that $\tilde{T}\colon W^{4,2}(\Sigma)\to L^2(\Sigma)\times Y$, with $Y$ as defined in \eqref{eq:spaces}, given by
    \begin{equation}
        \varphi \mapsto \tilde{T} \varphi\vcentcolon=\big((\Delta_g)^2\varphi - c_1\Delta_g\varphi + c_2\varphi, \partial_{\eta}\varphi + m \varphi \big|_{\partial \Sigma}, \partial_{\eta}\Delta_g\varphi + m \Delta_g\varphi + d\varphi.\vec b + b\varphi \big|_{\partial \Sigma} \big) 
    \end{equation}
    is an isomorphism.
\end{proposition}
\begin{proof}
    Consider the operator $K\colon W^{4,2}(\Sigma)\to L^2(\Sigma)\times Y$ with
    \begin{align}
        K^{(1)}\varphi&\vcentcolon=-c_1\Delta_g\varphi+c_2\varphi,\\
        K^{(2)}\varphi= m\varphi\big|_{\partial \Sigma}
        ,&\quad K^{(3)}\varphi\vcentcolon = m \Delta_g\varphi + d\varphi.\vec b + b\varphi  \big|_{\partial \Sigma}
    \end{align}
    where $c_1,c_2>0$ are fixed later. Note that $K$ 
    only depends on derivatives of $\varphi$ of up to second order. Combining \cite[Chapter 4, Propositions~4.4 and 4.5]{taylor2023}, $K$ is a compact operator.
    
    Since the operator $L$ in \Cref{thm:neumann-bilaplace} is Fredholm of index 0 and as $K$ is compact, using \cite[XVII, Corollaries~2.6 and 2.7]{lang1993}, $\tilde{T}=L+K$ is a Fredholm operator of index $0$. Therefore, in order to show that $\tilde{T}$ is an isomorphism, it is sufficient to verify injectivity of the operator. To this end, suppose that $\varphi\in W^{4,2}(\Sigma)$ satisfies $\tilde{T}\varphi=0$. That is,
    \begin{align}
        (\Delta_g)^2\varphi &= c_1\Delta_g\varphi - c_2\varphi\quad\text{on $\Sigma$,}\\
        \partial_{\eta}\varphi &= - m\varphi\quad\text{on $\partial\Sigma$ and}\\
        -\partial_{\eta}\Delta_g\varphi &= m \Delta_g\varphi + d\varphi.\vec b + b\varphi \quad\text{on $\partial\Sigma$}.
    \end{align}
    Integrating by parts, we get
    \begin{align}
        \int_{\Sigma} &|\Delta_g\varphi|^2\dd\mu = \int_{\partial \Sigma} \big(\Delta_g\varphi \ \partial_{\eta}\varphi -\partial_{\eta}\Delta_g\varphi \  \varphi\big)\dd s + \int_\Sigma (\Delta_g)^2\varphi \ \varphi\dd\mu\\
        &= \int_{\partial \Sigma} \Delta_g\varphi (-m\varphi) \dd s +\int_{\partial \Sigma} \varphi \big(m \Delta_g\varphi + d\varphi.\vec b + b\varphi\big)\dd s \\
        &\quad +c_1\int_\Sigma \Delta_g\varphi \ \varphi \dd\mu - c_2\int_\Sigma\varphi ^2\dd\mu\\
        &\leq C \int_{\partial\Sigma} \varphi^2 + |\varphi|\cdot|\nabla\varphi|\dd s +c_1\int_\Sigma \Delta_g\varphi \ \varphi \dd\mu - c_2\int_\Sigma\varphi^2\dd\mu
    \end{align}
    for a constant $C$ controlling the coefficients $\vec b$ and $b$ on $\partial\Sigma$ where $|\nabla\varphi|=|\nabla_g\varphi|_g$. Therefore, again integrating by parts, we have for all $\varepsilon\in(0,1)$
    \begin{align}
        \int_\Sigma |\Delta_g\varphi|^2 & +c_1|\nabla\varphi|^2+c_2\varphi^2\dd\mu \leq C \int_{\partial \Sigma} \varphi^2 + |\varphi|\cdot|\nabla\varphi| \dd s -c_1\int_{\partial \Sigma} m\varphi^2\dd s\\
        &\leq \varepsilon \|\varphi\|_{W^{1,2}(\partial \Sigma)}^2 + C (c_1+\frac{1}{\varepsilon})\|\varphi\|_{L^2(\partial \Sigma)}^2  \label{eq:inj-est}
    \end{align}
    for a constant $C$ changing from line to line, depending only on $L^\infty$-bounds on $\vec b$, $b$ and $m$. By \cite[Chapter 4, Proposition~4.5]{taylor2023}, for a constant additionally depending on the metric $g$,
    \begin{align}
        \int_\Sigma |\Delta_g\varphi|^2+c_1|\nabla\varphi|^2+c_2\varphi^2\dd\mu &\leq C \varepsilon \|\varphi\|_{W^{2,2}(\Sigma)}^2 + C (c_1+\frac{1}{\varepsilon})\|\varphi\|_{W^{1,2}(\Sigma)}^2 .\label{eq:inj-est-2}
    \end{align}
    Using that, by interpolation, we have
    \begin{equation}
        \|\varphi\|_{W^{1,2}(\Sigma)}^2 \leq C(g)\|\varphi\|_{W^{2,2}(\Sigma)} \|\varphi\|_{L^2(\Sigma)},
    \end{equation}
    \eqref{eq:inj-est-2} yields for a constant $C$ depending only on $g$ and on $L^\infty$-bounds on $\vec b$, $b$ and $m$ that, for all $\varepsilon\in(0,1)$,
    \begin{align}
        \int_\Sigma |\Delta_g\varphi|^2+c_1|\nabla\varphi|^2+c_2\varphi^2\dd\mu &\leq \varepsilon \|\varphi\|_{W^{2,2}(\Sigma)}^2 + \frac{C}{\varepsilon} (c_1^2+\frac{1}{\varepsilon^2})\|\varphi\|_{L^2(\Sigma)}^2 .\label{eq:inj-est-3}
    \end{align}
    By \cite[Chapter 5, Equation~(7.37)]{taylor2023},
    \begin{equation}
        \|\varphi\|_{W^{2,2}(\Sigma)}^2\leq C\big(\|\Delta_g\varphi\|_{L^2(\Sigma)}^2 + \|\partial_{\eta}\varphi\|_{W^{\frac12,2}(\partial \Sigma)}^2+\|\varphi\|_{W^{1,2}(\Sigma)}^2\big).
    \end{equation}    
    Again employing $\partial_{\eta}\varphi=-m\varphi$ on $\partial\Sigma$ and \cite[Chapter 5, Proposition~4.5]{taylor2023}, we have 
    \begin{equation}
        \|\varphi\|_{W^{2,2}(\Sigma)}^2\leq C\big(\|\Delta_g\varphi\|_{L^2(\Sigma)}^2 + \|\varphi\|_{W^{1,2}(\Sigma)}^2 \big)
    \end{equation}
    where $C$ depends on $g$ and $\|m\|_{C^1(\partial\Sigma)}$. Thus, first taking $\varepsilon$ to be sufficiently small, then choosing $c_1$ sufficiently large and finally, also depending on $\varepsilon$ and $c_1$, taking $c_2$ to be sufficiently large, \eqref{eq:inj-est-3} yields
    \begin{equation}
        \frac{1}{2} \int_\Sigma |\Delta_g\varphi |^2+|\nabla\varphi |^2+\varphi ^2\dd\mu \leq 0,
    \end{equation}
    so in particular that $\varphi=0$. This shows injectivity of $\tilde T$ with the above choices of $c_1$ and $c_2$. As $\tilde T$ is Fredholm of index $0$, the claim follows.
\end{proof}

\begin{proposition}\label{prop:submersion}
    The mapping $\mathcal{B}$ in \eqref{eq:boundcondi} is a submersion at $0$, that is, $\mathcal{B}'(0)\colon V\to Y$ is surjective. In particular, there exists $a_2>0$ such that, with $U_1$ as in \Cref{lem:def-fw},
    \begin{equation}
        U_2 \vcentcolon= \{w\in W^{4,2}(\Sigma) \mid \|w\|_{W^{4,2}}<a_2\} \subset U_1,
    \end{equation}
    and, for $V_0=\ker\mathcal{B}'(0)$, $\mathcal{M}=\mathcal{B}^{-1}(\{(0,0)\})\cap U_2$ is an analytic $V_0$-Banach manifold. Moreover, for all $w\in \mathcal{M}$, we have
    $T_w\mathcal{M} = \ker\mathcal{B}'(w)\subset W^{4,2}(\Sigma)$,    \begin{equation}\label{eq:deltaE}
        \E'(w)\varphi = \frac12 \int_{\Sigma} \delta\E(w) \cdot \varphi\langle\xi(\cdot,w(\cdot)),\nu_{f_w}\rangle^2\dd\mu_{f_w},
    \end{equation}
    for $\varphi\in \mathcal{T}_w\mathcal{M}$ and $\delta \E$ as in \eqref{eq:gradientWillmore}, and
    \begin{equation}\label{eq:loja-ass-1}
        \|\E'(w)\|_{(T_w\mathcal{M})^*} \leq C \|\delta\E(w)\|_{L^2(\dd\mu_{\bar{f}})} = C \|\delta\E(w)\|_{Z}.
    \end{equation}
\end{proposition}
\begin{proof} 
    By \Cref{lem:first-variations-B-and-E}, we can apply \Cref{prop:PDE-argument} with $g=g_{\bar f}$, $m=A^S(\nu_{\bar f},\nu_{\bar f})$ and 
    $\vec b=\vec b_1$, $b=b_0$. In particular, restricting the isomorphism in \Cref{prop:PDE-argument} to the $Y$-component, we obtain that $\mathcal B'(0)\colon V\to Y$ is surjective and thus that $\mathcal B$ is a submersion at $0$. The remainder of the statement now follows from \Cref{rem:on-submanifolds}.

    The proofs of \eqref{eq:deltaE} and \eqref{eq:loja-ass-1} essentially follow \cite{alessandronikuwert2016} and are postponed to \Cref{app:techproofs}.
\end{proof}

\begin{lemma}\label{lem:d^2E-0-with-B'0-is-fredholm}
     With $Y$ as defined in \eqref{eq:spaces}, consider the operator,
    \begin{equation}
        T\colon W^{4,2}(\Sigma)\to L^2(\Sigma)\times Y,\qquad T\varphi \vcentcolon= \big((\delta\E)'(0)\varphi, \mathcal{B}'(0)\varphi\big),
    \end{equation}
    and let $L$ be as in \Cref{thm:neumann-bilaplace}. Then $K=T-L$ is a compact operator, in particular, $T$ is Fredholm of index $0$.
\end{lemma}
\begin{proof}
    Using \eqref{eq:deltaeprime},
    $$        
        K^{(1)}\colon W^{4,2}(\Sigma)\to W^{1,2}(\Sigma)
    $$
    and, for $\varphi\in W^{4,2}(\Sigma)$, \eqref{eq:deltabprime} yields
    \begin{align}
        (K^{(2)}\varphi,K^{(3)}\varphi) = \big(A^S(\nu_{\bar{f}},\nu_{\bar{f}})\varphi\Big|_{\partial \Sigma},A^S(\nu_{\bar{f}},\nu_{\bar{f}})\Delta_{g_{\bar{f}}}\varphi + d\varphi.\vec b_1 + b_0\varphi\Big|_{\partial \Sigma}\big).
    \end{align}
    By the compact Sobolev embeddings \cite[Chapter 4, Proposition~4.4]{taylor2023} and the continuity of the trace \cite[Chapter 4, Proposition~4.5]{taylor2023}, $K$ is a compact operator. The claim then follows from \Cref{thm:neumann-bilaplace} and \cite[XVII, Corollaries~2.6 and 2.7]{lang1993}.
\end{proof}

\begin{proposition}\label{prop:deltaeprime-ffi}
    The mapping $(\delta\E)'(0)\colon T_0\mathcal{M}\to Z=L^2(\Sigma)$ is a Fredholm operator of finite index and thus satisfies assumptions (ii) and (iii) in \Cref{intro-thm:loja-abstract}.
\end{proposition}
\begin{proof}
    By \Cref{prop:submersion}, $T_0\mathcal{M}=\ker\mathcal{B}'(0)$. Consider the operator $T$ in \Cref{lem:d^2E-0-with-B'0-is-fredholm}. In order to show the claim, we only need to verify that the restriction $T|_{T_0\mathcal M}\colon T_0\mathcal{M}\to L^2(\Sigma)\times\{0\}$ is a Fredholm operator of finite index. Clearly, its kernel is given by $T_0\mathcal M\cap \ker T$ and thus still finite dimensional. Moreover, since $\mathrm{Im} (T)$ has finite codimension in $L^2(\Sigma)\times Y$, so does $\mathrm{Im}(T|_{T_0\mathcal M}) = \mathrm{Im}(T)\cap (L^2(\Sigma)\times\{0\})$ in $L^2(\Sigma)\times\{0\}$. 
\end{proof}

\begin{corollary}\label{cor:loja-gauss-coord}
    There exist $\theta\in (0,\frac12]$, $C'>0$ and $\sigma'>0$ such that, for all $w\in\mathcal{M}$ with $\|w\|_{W^{4,2}(\Sigma,\R)}<\sigma'$,
    \begin{equation}
        |\W(f_w)-\W(\bar{f})|^{1-\theta} \leq C' \|\nabla\W(f_w)\|_{L^2(\dd\mu_{\bar{f}})}.
    \end{equation}
\end{corollary}
\begin{proof}
    The result is a direct consequence of \Cref{intro-thm:loja-abstract}, using \Cref{lem:analytic}\eqref{item:analytic_4},\eqref{item:analytic_5}, \eqref{eq:loja-ass-1}, \Cref{prop:deltaeprime-ffi}, \eqref{eq:def-E} and \eqref{eq:gradientWillmore} with \eqref{eq:estimate-xi0fw-in-nu-fw}.
\end{proof}

For \Cref{intro-thm:main-result-fbc}, we have to extend the \L ojasiewicz--Simon gradient inequality \Cref{cor:loja-gauss-coord} to immersions $f$ satisfying \eqref{intro-eq:fbc} which are not necessarily parametrized in the Gaussian coordinates with respect to $\bar f$. The main tool is the following result on the existence of reparametrizations.

Recall that we assume $\Sigma\subset\R^M$, see \Cref{rem:dpartialsigma}.

\begin{proposition}\label{prop:reparam-as-fw}
    Let $a_2$ as in \Cref{prop:submersion} and $0<\Tilde{a}\leq a_2$. Then there exists $\tilde{\varepsilon}>0$ such that, for any immersion $f\in W^{4,2}(\Sigma,\R^3)$ with 
    \begin{equation}
        f(\partial \Sigma)\subset S\quad\text{and}\quad \|f-\bar{f}\|_{W^{4,2}(\Sigma)}<\tilde{\varepsilon},
    \end{equation}
    there exist $w\in W^{4,2}(\Sigma)$ with $\|w\|_{W^{4,2}}<\tilde{a}$ and a $C^1$-diffeomorphism $\Psi\in W^{4,2}(\Sigma,\R^M)$ of $\Sigma$ such that $f=f_w\circ \Psi$. Moreover, also $\Psi^{-1}\in W^{4,2}(\Sigma,\R^M)$.
\end{proposition}

\begin{remark}\label{rem:weak-inv-map}
    Since $\Sigma\subset\R^M$ is compact, proceeding as in \cite[Theorem~4.2]{campbellhenclkonopecky2015}, we have the following. If $\Psi\in W^{4,2}(\Sigma,\R^M)$ is a $C^1$-diffeomorphism of $\Sigma$, then also $\Psi^{-1}\in W^{4,2}(\Sigma,\R^M)$, so the last statement in \Cref{prop:reparam-as-fw} is clear.
\end{remark}

The proof of \Cref{prop:reparam-as-fw} is postponed to \Cref{app:proof-prop:reparam-as-fw}.
We may now prove \Cref{intro-thm:main-result-fbc}.

\begin{proof}[Proof of \Cref{intro-thm:main-result-fbc}]
    Fix $\theta$, $C'$ and $\sigma'$ as in \Cref{cor:loja-gauss-coord}. Then choose $\sigma\leq\tilde{\varepsilon}$ in \Cref{prop:reparam-as-fw} applied with $\tilde{a} < \min\{\sigma',a_2\}$. In particular, using \Cref{prop:reparam-as-fw} and \Cref{rem:weak-inv-map}, if $f\in W^{4,2}(\Sigma,\R^3)$ with $\|f-\bar{f}\|_{W^{4,2}(\Sigma)}<\sigma$, then there exists $w\in W^{4,2}(\Sigma,\R)$ with $\|w\|_{W^{4,2}(\Sigma)}<\sigma'$ and a $C^1$-diffeomorphism $\Psi\in W^{4,2}(\Sigma,\R^M)$ of $\Sigma$ such that
    \begin{equation}
        f\circ\Psi = f_w.
    \end{equation}
    Since \eqref{intro-eq:fbc} is invariant with respect to reparametrizations and since $\tilde a<a_2$, we have $w\in \mathcal{M}$ (see \Cref{prop:submersion}). Thus, with \Cref{cor:loja-gauss-coord}
    \begin{align}
        |\W(f)-\W(\bar{f})|^{1-\theta} &= |\W(f_w)-\W(\bar{f})|^{1-\theta} \leq C'\|\nabla\W(f_w)\|_{L^2(\dd\mu_{\bar{f}})} \\
        &\leq C \|\nabla\W(f)\circ\Psi\|_{L^2(\dd\mu_{f\circ\Psi})} = C \|\nabla\W(f)\|_{L^2(\dd\mu_{f})}
    \end{align}
    where we use the uniform equivalence of the metrics $g_{f_w}$ and $g_{\bar{f}}$ in the $W^{4,2}( \Sigma,\R^3)$-ball of radius $\sigma'$ centered at $\bar{f}$.
\end{proof}

\section{Stability of the free boundary Willmore flow}\label{sec:freebdy}

We will use the same notation as in \Cref{sec:loja_Willmore}. A family of immersions $f\colon[0,T)\times \Sigma\to\R^3$ satisfying
\begin{equation}\label{eq:free-bdry-willmore-flow}
    \begin{cases}
        \langle \partial_t f,\nu_f\rangle = -\big(\Delta_{g_f}H_f+|A^0_f|^2H_f\big)&\text{on $[0,T)\times \Sigma$}\\
        f(0,\cdot)=f_0&\text{on $\Sigma$}\\
        f(t,\partial \Sigma)\subset S &\text{for $0\leq t<T$}\\
        \langle \nu_f, N^S\circ f\rangle = 0 &\text{on $[0,T)\times\partial \Sigma$}\\
        \frac{\partial H_f}{\partial\eta_{f}} + A^S(\nu_f,\nu_f)H_f = 0&\text{on $[0,T)\times\partial \Sigma$}
    \end{cases}  
\end{equation}
is called \emph{free boundary Willmore flow} starting in $f_0$. 
Here we assume that the initial datum $f_0\colon \Sigma \to \R^3$  is a smooth immersion satisfying also \eqref{intro-eq:fbc}.

\subsection{Free boundary Willmore flow in Gaussian coordinates}

Throughout this section, we fix a Willmore immersion $\bar f\colon \Sigma\to\R^3$ satisfying \eqref{intro-eq:fbc} and use this as the reference immersion for our Gaussian coordinate system $f=f_w$ near $\bar{f}$, see \Cref{lem:def-fw}. In particular, also recall $\vec\mu$ from \Cref{lem:reformulation-first-order-bdry-condition}, $\delta\E$ as defined in \eqref{eq:gradientWillmore} and the reformulated boundary conditions in \eqref{eq:boundcondi}.

\begin{lemma}\label{lem:fbwf-in-gauss-coord}
    Let $w\colon[0,T)\times \Sigma\to(-\bar r,\bar r)$ 
    smooth with $w(t)\in U_1\footnote{as in \Cref{lem:def-fw}}$ for all $0\leq t<T$ and set
     $f\vcentcolon=f_w$. Then $w$ solves
     \begin{equation}\label{eq:schauder-2}
        \begin{cases}
            \partial_t w = - \frac{1}{\langle\xi(\cdot,w(\cdot)),\nu_{f_w}\rangle}\big(\Delta_{g_{f_w}}H_{f_w}+|A^0_{f_w}|^2H_{f_w}\big) &\text{on $[0,T)\times \Sigma$}\\
            w(0,\cdot)=w_0&\text{on $\Sigma$}\\
            \partial_{\vec\mu^{(1)}(x,w(t,x))}w(x)-\vec\mu^{(2)}(x,w(t,x)) =0&\text{for $(t,x)\in [0,T)\times\partial \Sigma$}\\
            \frac{\partial H_{f_w}}{\partial \eta_{f_w}} - (\nabla^2 d^S)\circ f_w~(\nu_{f_w},\nu_{f_w})H_{f_w} = 0 &\text{on $[0,T)\times\partial \Sigma$}
        \end{cases}
    \end{equation}
     if and only if $f$ is a free boundary Willmore flow starting in $f_{w_0}$.    
\end{lemma}
\begin{proof}
    Consider a smooth family $w\colon[0,T)\times \Sigma\to (-\bar r,\bar r)$, set $f\colon[0,T)\times \Sigma\to \R^3$, $f(t,x)=f_{w(t,\cdot)}(x)=\Phi(x,w(t,x))$. Using that 
    \begin{equation}
        \partial_tw\cdot\langle \xi(\cdot,w(\cdot)),\nu_{f_w}\rangle  = \langle \partial_tf,\nu_f\rangle,
    \end{equation}
    the claim follows using that the boundary conditions are equivalent by \Cref{lem:def-fw,lem:reformulation-first-order-bdry-condition}.
\end{proof}

To study short-time existence and apply Schauder estimates for the quasilinear parabolic boundary problem \eqref{eq:schauder-2}, we rewrite the PDE and the boundary conditions as follows. 

First, given two metrics $g_1$ and $g_2$ on $\Sigma$, where $g_1\in C^0$ and $g_2\in C^2$, we define the differential operator $\mathbb{D}(g_1,g_2)\vcentcolon= \mathrm{tr}_{g_1}\nabla^2_{g_2}$, mapping from $W^{2,2}(\Sigma,\R)$ to $L^2(\Sigma,\R)$. Note that $\mathbb{D}(g_2,g_2)=\Delta_{g_2}$ is the Laplace--Beltrami operator. This allows us to identify the PDE and the boundary operators as quasilinear differential operators of the appropriate orders and lower order terms.

For the PDE in \eqref{eq:schauder-2}, let $w_1, w_2\in U_1$ and define $\mathbb A(w_1).w_2 \vcentcolon= \mathbb{D}(g_{f_{w_1}},g_{\bar f})^2.w_2$. In local coordinates $x=(x^1,x^2)$, we have
\begin{align}
    \mathbb A(w_1).w_2 &= g_{f_{w_1}}^{ij}g_{f_{w_1}}^{kl} \partial^4_{ijkl} w_2 + \sum_{0<|\alpha|\leq 3} b_{\alpha,0}\big(\bar \Gamma,\partial \bar \Gamma,\partial^2 \bar \Gamma,g_{f_{w_1}},\partial g_{f_{w_1}},\partial^2 g_{f_{w_1}}\big)\partial^\alpha w_2 \label{eq:local-coords-mathbb-A}
\end{align}
for smooth functions $b_{\alpha,0}$ where $\bar\Gamma = \bar\Gamma^k_{ij}$ denotes the Christoffel symbols of $g_{\bar f}$ and $\partial$ denotes differentiation in the local coordinates. Moreover, a direct computation yields
\begin{equation}\label{eq:def-F0}
    \frac{1}{\langle\xi(\cdot,w(\cdot)),\nu_{f_w}\rangle}\big(\Delta_{g_{f_w}}H_{f_w}+|A^0_{f_w}|^2H_{f_w}\big) - \mathbb A(w).w = F_0(\cdot,w,\bar \nabla w,\bar\nabla^2 w,\bar\nabla^3w)
\end{equation}
for a smooth function $F_0$ where $\bar\nabla$ denotes covariant differentiation with respect to $g_{\bar f}$. Furthermore, letting 
\begin{align}
    \mathbb B_1(w_1).w_2 &\vcentcolon= \partial_{\vec\mu^{(1)}(\cdot,w_1(\cdot))}w_2|_{\partial\Sigma},\\
    \mathbb B_2(w_1).w_2 &\vcentcolon= \partial_{\eta_{f_{w_1}}}\mathbb{D}(g_{f_{w_1}},g_{\bar f}).w_2 |_{\partial\Sigma},
\end{align}
we have the following. Let $w_1\in U_1$ such that $\partial_{\eta_{f_{w_1}}}f_{w_1} = N^S\circ f_{w_1}$ on $\partial\Sigma$, and let $w_2\in U_1$. Then, in local coordinates $x=(x^1,x^2)\colon U\overset{\approx}{\to}\Omega\subset \R\times[0,\infty)$ in a neighborhood $U$ of a point in $\partial\Sigma$,
using that $\vec\mu^{(1)}(\cdot,w_1(\cdot))=\eta_{f_{w_1}}$ by \Cref{lem:loc-coord-w-loja-sec} and \eqref{eq:eta_g-expl-form}, we find
\begin{align}
    \mathbb B_1(w_1).w_2 &= - (g^{22}_{f_{w_1}})^{-\frac12}g_{f_{w_1}}^{2i}\partial_iw_2,\\
    \mathbb B_2(w_1).w_2 &= -(g^{22}_{f_{w_1}})^{-\frac12}g_{f_{w_1}}^{2i}g_{f_{w_1}}^{kl}\partial^3_{ikl}w_2 + \sum_{0<|\alpha|\leq 2} b_{\alpha,2}\big(\bar\Gamma,\partial \bar \Gamma,g_{f_{w_1}},\partial g_{f_{w_1}}\big)\partial^\alpha w_2 \label{eq:loc-coords-mathbb-Bi}
\end{align}
for suitable smooth functions $b_{\alpha,2}$. Moreover, for $w\in U_1$, we have
\begin{align}
    & \partial_{\vec\mu^{(1)}(\cdot,w(\cdot))}w-\vec\mu^{(2)}(\cdot,w(\cdot)) -\mathbb B_1(w).w = F_1(\cdot,w(\cdot))\quad\text{and} \\
    &\frac{1}{\langle\xi(\cdot,w(\cdot)),\nu_{f_w}\rangle}\big(\partial_{\eta_{f_w}} H_{f_w} - (\nabla^2 d^S)\circ f_w~(\nu_{f_w},\nu_{f_w})H_{f_w}\big)-\mathbb B_2(w).w \\
    &\qquad = F_2(\cdot,w,\bar \nabla w,\bar\nabla^2 w)\label{eq:def-F1-F2}
\end{align}
for smooth functions $F_1,\ F_2$.

\begin{remark}\label{rem:good-atlas}
    We refer to \cite[Section~1.3]{metsch2023} for the definition of (parabolic) Hölder spaces on $\Sigma$. Especially, consider a \emph{good atlas} $\mathcal{A}=\{(U_i,\varphi_i)\}_{i=1}^N$ of $\Sigma$, that is, $V_i=\varphi_i(U_i)$ are convex and open in $\R\times[0,\infty)$, and the charts can be smoothly extended. 
    Then $C^{k,\gamma}(\Sigma)=C^{k,\gamma}_{\mathcal{A}}(\Sigma)$ can be characterized in the local coordinates of this good atlas, see \cite[Lemma~1.3.3]{metsch2023}, and similarly for the parabolic Hölder spaces $C^{\frac{k+\gamma}4,k+\gamma}([0,T]\times\Sigma)$, see \cite[Lemma~1.3.5]{metsch2023} where $k\in\N_0$ and $\gamma\in(0,1)$.
\end{remark}

\begin{proposition}
\label{prop:ste-at-0}
    Let $\alpha\in (0,1)$, $T_0>0$ and $\delta>0$ be arbitrary. There exists $\varepsilon=\varepsilon(\alpha,\bar{f},\delta,T_0)>0$ with the following property. For all $w_0\in C^{4,\alpha}(\Sigma)$ with $\|w_0\|_{C^{4,\alpha}(\Sigma)}<\varepsilon$ such that $f_{w_0}$ satisfies \eqref{intro-eq:fbc}, there exists $w\in C^{\frac{4+\alpha}{4},4+\alpha}([0,T_0]\times \Sigma)$ solving \eqref{eq:schauder-2} and satisfying 
    \begin{equation}
        \|w\|_{C^{\frac{4+\alpha}{4},4+\alpha}([0,T_0]\times \Sigma)} < \delta.
    \end{equation}
\end{proposition}
\begin{proof}
    Write
    \begin{align}
        \mathbb{X} &\vcentcolon= C^{\frac{4+\alpha}{4},4+\alpha}([0,T_0]\times \Sigma),\\
        \mathbb{Y} &\vcentcolon= C^{\frac{\alpha}{4},\alpha}([0,T_0]\times \Sigma) \times C^{4,\alpha}(\Sigma) \times C^{\frac{3+\alpha}{4},3+\alpha}([0,T_0]\times \partial \Sigma) \times C^{\frac{1+\alpha}{4},1+\alpha}([0,T_0]\times \partial \Sigma),\\
        \mathcal{U} &\vcentcolon= \{w\in\mathbb{X}\mid w(t)\in U_2 \text{ for all $0\leq t\leq T_0$}\}
    \end{align}
    with $U_2$ as in \Cref{prop:submersion}. Consider $\mathcal{B}\colon C^{4,\alpha}(\Sigma)\to C^{3,\alpha}(\partial \Sigma)\times C^{1,\alpha}(\partial \Sigma)$, 
    \begin{align}
        \mathcal B(w) \vcentcolon= \big(&\partial_{\vec\mu^{(1)}(\cdot,w(\cdot))}w -\vec\mu^{(2)}(\cdot,w(\cdot)) \Big|_{\partial\Sigma}, \\
    &\frac{1}{\langle \xi(\cdot,w(\cdot)),\nu_{f_w}\rangle}\Big[\frac{\partial H_{f_w}}{\partial \eta_{f_w}} - (\nabla^2 d^S)\circ f_w~(\nu_{f_w},\nu_{f_w})H_{f_w}\Big]\Big|_{\partial \Sigma}\big) \\
    &=\big(\mathbb B_1(w).w+F_1(w),\mathbb B_2(w).w+F_2(w)\big)
    \end{align}
    and let $\mathcal{M}_{\mathbb Y}\vcentcolon=\{(f,w_0,h_1,h_2)\in\mathbb Y\mid \mathcal{B}(w_0)=(h_1(0,\cdot),h_2(0,\cdot))\}$. 
    It is easy to verify that $(f,w_0,h_1,h_2)\mapsto \mathcal{B}(w_0)-(h_1(0,\cdot),h_2(0,\cdot))$ is a submersion at $(0,0,0,0)$, proving surjectivity of the linearization by using that $(h_1,h_2)\mapsto - (h_1(0,\cdot),h_2(0,\cdot))$ already is surjective.
    Therefore, by \Cref{rem:on-submanifolds} there exists $\varepsilon'>0$ such that $\mathcal{M}_{\mathbb Y}\cap B_{\varepsilon'}((0,0,0,0)) \subset\mathbb Y$ is a Banach (sub-)manifold with
    \begin{equation}\label{eq:lin-comp-cond}
        T_{(0,0,0,0)}(\mathcal M_{\mathbb Y}\cap B_{\varepsilon'}(0)) = \{(f,w_0,h_1,h_2) \in \mathbb Y\mid \mathcal{B}'(0).w_0-(h_1(0,\cdot),h_2(0,\cdot))=0\}. 
    \end{equation}
    Then consider the operator $\mathcal{T}\colon \mathcal{U} \to \mathbb{Y}$ with
    \begin{align}
        \mathcal{T}(w) \vcentcolon= \Big(&\partial_t w +  \frac{1}{\langle\xi(\cdot,w(\cdot)),\nu_{f_w}\rangle}\big(\Delta_{g_{f_w}}H_{f_w}+|A^0_{f_w}|^2H_{f_w}\big),w(0,\cdot ), t\mapsto \mathcal{B}(w(t,\cdot)) \Big).
    \end{align} 
    In particular, we have that $\mathcal{T}(\mathcal U)\subset \mathcal M_{\mathbb Y}$ and, without loss of generality making $\delta$ smaller, $\mathcal{T}(B_{\delta}(0))\subset \mathcal M_{\mathbb Y}\cap B_{\varepsilon'}((0,0,0,0))$, using that $\mathcal{T}(0)=0$ since $\bar{f}=f_w|_{w=0}$ is a Willmore immersion satisfying the free boundary conditions \eqref{intro-eq:fbc}, cf.\ \Cref{lem:choice-of-bdryconditions}. 
    We then have that
    \begin{equation}\label{eqcl:t-prime-0-invertible}
        \mathcal{T}\in C^1(\mathcal{U},\mathbb{Y}) \quad\text{and}\quad \mathcal{T}'(0)\colon\mathbb{X}\to T_{(0,0,0,0)}(\mathcal M_{\mathbb Y}\cap B_{\varepsilon'}(0)) \text{ is invertible}.
    \end{equation}
    A detailed justification for \eqref{eqcl:t-prime-0-invertible} goes beyond the scope of this article. We only point out that, using \eqref{eq:def-F0} and \eqref{eq:def-F1-F2}, the principal part of $\mathcal T'(0)$ is determined by the linear operators $\mathbb A(0)$, $\mathbb B_1(0)$ and $\mathbb B_2(0)$. Thus, using \eqref{eq:local-coords-mathbb-A} and \eqref{eq:loc-coords-mathbb-Bi}, statement \eqref{eqcl:t-prime-0-invertible} can be proved using  linear parabolic theory. Hereby, the linear compatibility conditions for the data are satisfied due to \eqref{eq:lin-comp-cond}. 

    By \eqref{eqcl:t-prime-0-invertible}, the inverse function \Cref{thm:inv-fct} applies and thus, there are neighborhoods $\mathcal{U}_0\subset \mathcal{U}$ of $0$ and $\mathbb{Y}_0\subset B_{\varepsilon'}((0,0,0,0))\subset \mathbb{Y}$ of $0$ such that $\mathcal{T}$ restricts to a diffeomorphism $\mathcal{U}_0\to\mathbb{Y}_0\cap\mathcal M_{\mathbb Y}$. Without loss of generality, $\delta$ is sufficiently small such that $B_{\delta}(0)\subset \mathcal{U}_0$. Then choose $0<\varepsilon<\varepsilon'$ such that
    \begin{equation}
        \mathcal M_{\mathbb Y}\cap (\{0\}\times B_{\varepsilon}(0) \times \{0\}\times \{0\}) \subset \mathcal{T}(B_{\delta}(0)).
    \end{equation}
    The claim follows setting $w=\mathcal{T}^{-1}(0,w_0,0,0)$, see \eqref{eq:schauder-2}.
\end{proof}

\subsection{Stability of local minimizers under the flow}

In this subsection, we prove the following stability result.

\begin{theorem}\label{lem:near-crit-conv}
    Let $\bar{f}\colon \Sigma\to\R^3$ be a Willmore immersion satisfying \eqref{intro-eq:fbc}. Let $\delta>0$. Then there exists $\varepsilon=\varepsilon(\bar{f},\delta)>0$ with the following property. If $f\colon [0,T)\times \Sigma\to\R^3$ is a free boundary Willmore flow with $f=f_w$ where $w\colon[0,T)\times \Sigma\to\R$ is a maximal solution of \eqref{eq:schauder-2}, and if
		\begin{enumerate}[(i)]
			\item $\Vert {f_0-\bar{f}}\Vert_{C^{4,\alpha}(\Sigma)}<\varepsilon$ for some $\alpha>0$;
			\item $\W(f(t))\geq\W(\bar{f})$ whenever $\Vert f(t) - \bar{f}\Vert_{C^4(\Sigma)}\leq \delta$;
		\end{enumerate}
		then $T=\infty$. Moreover, as $t\to\infty$, $f(t)$ converges smoothly 
        to a Willmore immersion $f_\infty$, satisfying $\W(f_\infty)=\W(\bar{f})$ and $\Vert {f_{\infty}-\bar{f}}\Vert_{C^4(\Sigma)}\leq\delta$.        
\end{theorem}

A key ingredient are the following Schauder-type estimates.
\begin{lemma}
\label{lem:schauder}
    Let $T_0>0$, $T\in [T_0,\infty)$ and   let
    $w\in C^{\frac{4+\alpha}{4},4+\alpha}([0,T']\times \Sigma)$ for all $T'<T$ be a solution of \eqref{eq:schauder-2}. There exist $\alpha'\in(0,\alpha)$ and $\varepsilon_1>0$ with the property that, if 
    \begin{equation}\label{eq:schauder-1}
        \|f_{w(t)}-\bar{f}\|_{C^4(\Sigma)} \leq \varepsilon_1\quad\text{for all $0\leq t<T$}
    \end{equation}
    and letting $M>0$ such that $\|w_0\|_{C^{4,\alpha}(\Sigma)}\leq M$, then
    \begin{equation}
        \|w\|_{C^{\frac{4+\alpha'}{4},4+\alpha'}([0,T]\times \Sigma)} \leq C(T_0,\bar{f},\alpha,M).
    \end{equation}
\end{lemma}
\begin{proof}
    Consider the linear differential operator 
    \begin{equation}
        \mathcal P(u)\vcentcolon=\partial_tu +\mathbb A(w(t,\cdot)).u(t,\cdot)
    \end{equation}
    for functions $u=u(t,x)\colon [0,T]\times \Sigma\to (-\bar r,\bar r)$,
    and the linear boundary operators
    \begin{align}
        \mathcal{B}_1(u)&\vcentcolon=\mathbb B_1(w(t,\cdot)).u(t,\cdot),\\
        \mathcal{B}_2(u)&\vcentcolon=\mathbb B_2(w(t,\cdot)).u(t,\cdot).
    \end{align}
    Using \eqref{eq:local-coords-mathbb-A} and \eqref{eq:loc-coords-mathbb-Bi}, the boundary conditions $\mathcal{B}$ are $\mathcal{P}$-compatible in the sense of \cite[Definition~4.1.4]{metsch2023}.
    
    Choosing $\varepsilon_1>0$ sufficiently small and using that $(g^{ij}_{\bar{f}})_{i,j=1}^2$ is positive definite, one finds that  $\mathcal{P}$ is strictly elliptic, uniformly in $(t,x)$. That is, for a good atlas $\mathcal A$ as in \Cref{rem:good-atlas}, there exists $\delta>0$ such that, in the local coordinates of all charts $(U_i,\varphi_i)\in\mathcal A$, 
    \begin{equation}
        -g^{ij}_{f_{w(t)}}(x)g^{kl}_{f_{w(t)}}(x)\varsigma^{(i)}\varsigma^{(j)}\varsigma^{(k)}\varsigma^{(l)} \leq -\delta |\varsigma|^4\quad\text{for all $0\leq t<T$ and $x\in U_i$}.
    \end{equation}

    Successively applying $\eqref{eq:schauder-1}$ and \eqref{eq:schauder-2}, 
    \begin{equation}\label{eq:schauder-4}
        \sup_{0\leq t<T} \|w\|_{C^4(\Sigma)} \leq C(\bar{f}) \quad\text{and}\quad \sup_{0\leq t<T} \|\partial_tw\|_{C^0(\Sigma)} \leq C(\bar{f}).
    \end{equation}
    Since
    \begin{equation}\label{eq:schauder-3}
        \mathrm{Lip}(0,T;C^0(\Sigma))\cap \mathrm{Bd}(0,T;C^4(\Sigma))\hookrightarrow C^{\vartheta}(0,T;C^{4(1-\vartheta)}(\Sigma))
    \end{equation}
    for all $\vartheta\in(0,1)$ by \cite[Propositions~1.1.3 and 1.1.4]{lunardi1995} where the embedding constant does not depend on $T$ (see \cite[Equation~(1.1.3)]{lunardi1995}), one concludes for the coefficients of $\mathcal{P}$ and $\mathcal{B}_2$ the following, cf.\ \eqref{eq:local-coords-mathbb-A}, \eqref{eq:loc-coords-mathbb-Bi}. There exists $\alpha'\in(0,1)$ such that, in the local coordinates of any chart $(U_n,\varphi_n)\in\mathcal A$, any of the norms    
    \begin{align}
        &\|g^{ij}_{f_{w}}g^{kl}_{f_{w}}\|_{C^{\frac{\alpha'}{4},\alpha'}([0,T)\times U_n)},\qquad 
        \|b_{\alpha,0}\big(\bar \Gamma,\partial \bar \Gamma,\partial^2 \bar \Gamma,g_{f_{w}},\partial g_{f_{w}},\partial^2 g_{f_{w}}\big)\|_{C^{\frac{\alpha'}{4},\alpha'}([0,T)\times U_n)}       
    \end{align}
    can be bounded by a constant $C=C(\bar{f},\alpha',\mathcal A)$. Similarly, in any chart $(U_n,\varphi_n)\in\mathcal A$ around 
    a point in $\partial\Sigma$, any of the norms
    \begin{align}
        &\|(g^{22}_{f_w})^{-\frac12}g^{2i}_{f_w}g^{kl}_{f_w}\|_{C^{\frac{1+\alpha'}{4},1+\alpha'}([0,T)\times (U_n\cap\partial\Sigma))},\\
        &\|b_{\alpha,2}\big(\bar\Gamma,\partial \bar\Gamma,g_{f_{w}},\partial g_{f_{w}}\big)\|_{C^{\frac{1+\alpha'}{4},1+\alpha'}([0,T)\times (U_n\cap\partial\Sigma))}
    \end{align}
    can be bounded by a constant $C=C(\bar{f},\alpha',\mathcal A)$. Regarding $\mathcal B_1=\partial_{\vec\mu^{(1)}(\cdot,w(\cdot))}$, we have
    \begin{equation}
        \|\vec\mu^{(1)}(\cdot,w)\|_{C^{\frac{3+\alpha'}{4},3+\alpha'}([0,T)\times \partial\Sigma)} \leq C(\bar f,\alpha',\mathcal A).
    \end{equation}  
    By the parabolic Schauder estimates in \cite[Theorem~4.1.6]{metsch2023} for the problem determined by $\mathcal{P}$ and $\mathcal{B}$, without loss of generality taking $0<\alpha'<\alpha$, for a constant $C=C(T_0,\alpha',\bar{f})$
    \begin{align}
        \|w\|_{C^{\frac{4+\alpha'}{4},4+\alpha'}([0,T]\times\Sigma)}  &\leq C\Big(\|F_0(\cdot,w,\bar\nabla w,\bar\nabla^2 w,\bar\nabla^3w)\|_{C^{\frac{\alpha'}{4},\alpha'}([0,T]\times\Sigma)}+\|w_0\|_{C^{4,\alpha'}(\Sigma)}\\
        &\qquad + \|F_1(\cdot,w)\|_{C^{\frac{3+\alpha'}{4},3+\alpha'}([0,T]\times\partial\Sigma)}  \\
        &\qquad + \| F_2(\cdot,w,\bar\nabla w,\bar\nabla^2 w)\|_{C^{\frac{1+\alpha'}{4},1+\alpha'}([0,T]\times\partial\Sigma)}\\
        &\qquad+\sup_{0\leq t\leq T}\|w(t,\cdot)\|_{L^2(\Sigma)}\Big),
    \end{align}
    using \eqref{eq:def-F0} and \eqref{eq:def-F1-F2}. Note that the right-hand side is bounded by a constant $C=C(T_0,\bar{f},\alpha',\mathcal A)$, using \eqref{eq:schauder-4} and \eqref{eq:schauder-3}. This proves the claim.
\end{proof}

\begin{proof}[Proof of \Cref{lem:near-crit-conv}]
    Writing $w_0=w(0,\cdot)$, assumption (i) yields 
    \begin{equation}
        \|w_0\|_{C^{4,\alpha}(\Sigma)}\leq C\varepsilon
    \end{equation}
    for some constant $C$ only depending on $\bar f$ and $S$. Fix $T_0>0$.
    With $\varepsilon_1$ as in \Cref{lem:schauder}, let $0<\tilde\sigma<\min\{\delta,\varepsilon_1\}$ be sufficiently small such that $\|h-\bar{f}\|_{C^4(\Sigma)}<\tilde\sigma$ implies $\|h-\bar{f}\|_{W^{4,2}(\Sigma)}<\sigma(\bar{f})$ with $\sigma(\bar{f})$ as in \Cref{intro-thm:main-result-fbc}.
    By \Cref{prop:ste-at-0}, choosing $\varepsilon=\varepsilon(\tilde\sigma)$ sufficiently small, 
    we have $T>T_0$ and there exists 
    $T_{\max}\in [T_0,T]$ such that
    \begin{equation}\label{eq:choice-tilde-sigma}
        \|{f}(t)-\bar{f}\|_{C^{4}(\Sigma)} \leq \tilde\sigma < \delta
    \end{equation}
    on a (with respect to \eqref{eq:choice-tilde-sigma}) maximal time interval $t\in [0,T_{\max})$. 
    Using \Cref{lem:schauder} and $\tilde\sigma<\varepsilon_1$, there are constants $C=C(T_0,\bar{f})>0$, $\alpha'\in (0,\alpha)$ with
    \begin{equation}\label{eq:schauder-bound}
        \sup_{0\leq t<T_{\max}} \|{f}(t)-\bar{f}\|_{C^{4,\alpha'}(\Sigma)} \leq C(T_0,\bar{f},\alpha).
    \end{equation}
    Without loss of generality, we may assume that $\W(\tilde{f}(t))>\W(\bar{f})$ for all $0\leq t<T_{\max}$. Let $\theta$ and $C_L$ be as in the \L ojasiewicz-Simon gradient inequality \Cref{intro-thm:main-result-fbc}. Then
    \begin{align}
        -\partial_t \big(\W({f}(t))-\W(\bar{f})\big)^{\theta} &= -2\theta \big(\W({f}(t))-\W(\bar{f})\big)^{\theta-1} \int_{\Sigma} \langle\nabla\W({f}),\partial_t{f}\rangle\dd\mu_{\tilde{f}}\\
        &= 2\theta \big(\W({f}(t))-\W(\bar{f})\big)^{\theta-1} \|\nabla\W({f})\|_{L^2(\dd\mu_{\tilde{f}})} \|\partial_t^{\bot}{f}\|_{L^2(\dd\mu_{{f}})}\\
        &\geq \frac{\theta}{C_L} \|\partial_t^{\bot}{f}\|_{L^2(\dd\mu_{{f}})}.
    \end{align}
    Since $\partial_t^{\bot}{f}=\partial_t^{\bot}f_w=\langle\xi(\cdot,w(\cdot)),\nu_{f_w}\rangle\partial_tw\cdot \nu_{f_w}$ and $\partial_t f = \xi(\cdot,w(\cdot))\  \partial_tw$, and using \eqref{eq:estimate-xi0fw-in-nu-fw}, we have for all $0\leq t<T_{\max}$
    \begin{equation}
        \|\partial_t{f}\|_{L^2(\dd\mu_{ f})} \leq C \|\partial_t w\|_{L^2(\dd\mu_{ f})} \leq -C\partial_t\big(\W({f}(t))-\W(\bar {f})\big)^{\theta}
    \end{equation}
    for some constant $C=C(\bar f)>0$ changing from term to term. Since $g_{ f}$ is uniformly equivalent to $g_{\bar f}$ on $[0,T_{\max})$ by \eqref{eq:choice-tilde-sigma}, and since $\mu_{\bar f}$ is equivalent to the Riemannian measure $\mu_{g_0}$ on $\Sigma$ induced by the reference metric $g_0$ (see \Cref{rem:dpartialsigma}),  
    \begin{equation}\label{eq:loja_along_flow}
        \|\partial_t{f}\|_{L^2(\Sigma)} \leq -C\partial_t\big(\W({f}(t))-\W(\bar{f})\big)^{\theta}
    \end{equation}
    for some constant $C=C(\bar f)>0$. Integrating in time yields for every $0\leq t<T_{\max}$
    \begin{equation}\label{eq:appl-loja}
        \|{f}(t)-\bar{f}\|_{L^2(\Sigma)} \leq \|f_0-\bar{f}\|_{L^2(\Sigma)} + C\big(\W(f_0)-\W(\bar{f})\big)^{\theta} \leq C\|f_0-\bar{f}\|_{C^2(\Sigma)}^{\theta}.
    \end{equation}
    For some $\beta\in (0,1)$ suitably small and depending only on $\alpha'$, by real interpolation, for a constant $C$ changing from line to line depending only on $\alpha$, $T_0$ and $\bar{f}$,
    \begin{align}
        \|{f}(t)-\bar{f}\|_{C^4(\Sigma)}&\leq C \|{f}(t)-\bar{f}\|_{C^{4,\alpha'}(\Sigma)}^{1-\beta}\|{f}(t)-\bar{f}\|_{L^2(\Sigma)}^{\beta}\\
        &\leq C \|f_0-\bar{f}\|_{C^2(\Sigma)}^{\beta\theta} \leq C\varepsilon^{\beta\theta}, \quad\text{for every $0\leq t<T_{\max}$},\label{eq:loja_along_flow_C4}
    \end{align}
    using \eqref{eq:schauder-bound} and \eqref{eq:appl-loja}. Making $\varepsilon=\varepsilon(\alpha,\beta,T_0,\bar{f},\tilde\sigma)$ even smaller such that also $C\varepsilon^{\beta\theta}<\frac12\tilde\sigma$, the maximality of $T_{\max}$ with respect to \eqref{eq:choice-tilde-sigma} yields $T_{\max}=\infty$.
    Continuing as in \cite[pp.~361 -- 362]{chillfasangovaschaetzle2009} yields the claim with arguments similar to those above.
\end{proof}

With \Cref{lem:near-crit-conv} at hand, we can finally turn to the proofs of \Cref{intro-thm:FBWF} and \Cref{cor:quanti_stability}.

\begin{proof}[Proof of \Cref{intro-thm:FBWF}]
Since $\bar f$ is a local minimizer of $\mathcal{W}$ among immersions with \eqref{intro-eq:fbc}, it is a free boundary Willmore immersion. Let $\bar \varepsilon>0$ to be chosen and fix a smooth immersion $f_0$ with \eqref{intro-eq:fbc}. Let $\varepsilon>0$ be as in \Cref{prop:ste-at-0} for $\delta=1, T_0=1$. \Cref{prop:reparam-as-fw} and \Cref{rem:reparam_fw_C4a} yield that for $\bar \varepsilon>0$ sufficiently small, $f_0=f_{w_0}\circ \Psi$ for some $w_0\in C^{4,\alpha}(\Sigma)$ with $\Vert w_0\Vert_{C^{4,\alpha}}<\varepsilon$
and a $C^{4,\alpha}$-diffeomorphism $\Psi$. Now, \Cref{prop:ste-at-0} and \Cref{lem:fbwf-in-gauss-coord} yield the existence of a solution $f(t)$, $0\leq t\leq 1$, to the free boundary Willmore flow starting from $f_{w_0}$. Reducing $\bar \varepsilon>0$, we may further assume that the assumptions of \Cref{lem:near-crit-conv} are satisfied  (with $\delta>0$ such that $\bar f$ is a local minimizer). It follows from \Cref{lem:near-crit-conv} that $f(t)$ exists globally and converges smoothly as $t\to\infty$ to some free boundary Willmore immersion $f_\infty$ with $\mathcal{W}(f_\infty)=\mathcal{W}(\bar f)$.
\end{proof}

\begin{proof}[Proof of \Cref{cor:quanti_stability}]
The assumptions allow us to apply \Cref{lem:near-crit-conv}, so the free boundary Willmore flow $f(t)$ starting at $f_{w_0}=f_0\circ\Psi$ as in the proof of \Cref{intro-thm:FBWF} exists globally and \eqref{eq:loja_along_flow} is valid for all $t\in [0,\infty)$. With $C=C(\bar f)>0$ and $\theta=\theta(\bar f)\in (0,\frac12]$ as in \eqref{eq:loja_along_flow}, we thus conclude that
\begin{align}\label{eq:quanti_stab_1}
    \Vert f(t)-f_{w_0}\Vert_{L^2(\Sigma)} \leq C \big(\mathcal{W}(f_0)-\mathcal{W}(\bar f)\big)^\theta.
\end{align}
Now, arguing as in \eqref{eq:loja_along_flow_C4} with $T_0=1$ and $\bar f$ replaced by $f_{w_0}$, we find
\begin{align}\label{eq:quanti_stab_2}
\Vert f(t)-f_{w_0}\Vert_{C^4(\Sigma)}
&\leq C \Vert f(t)-f_{w_0}\Vert_{L^2(\Sigma)}^\beta,
\end{align}
for some $\beta\in (0,1)$ and $C=C(\alpha, \bar f)$ determined by \eqref{eq:schauder-bound}. Combining \eqref{eq:quanti_stab_1} and \eqref{eq:quanti_stab_2} and sending $t\to\infty$ yields the statement.
\end{proof}

\section*{Acknowledgments}
This research was funded in whole, or in part, by the Austrian Science Fund (FWF), grant numbers \href{https://doi.org/10.55776/ESP557}{10.55776/ESP557} and \href{https://doi.org/10.55776/PAT8894924}{10.55776/PAT8894924}. Part of this work was conducted while the first two authors were visiting the \emph{Erwin Schrödinger International Institute for Mathematics and Physics} within the thematic program \emph{Free Boundary Problems}.


\appendix

\section{Banach manifolds}\label{sec:banach_manifolds}

For abstract Banach manifolds, we follow \cite{lang1972,abrahammarsdenratiu1988}.

\begin{definition}
    Let $X$ be a set. An atlas of class $C^p$ on $X$ where $p\in\N\cup\{\omega\}$ (here $C^\omega$ denotes \emph{analytic}) is a collection of pairs $(U_i,\varphi_i)$ satisfying
    \begin{enumerate}[(a)]
        \item $U_i\subset X$ and $\bigcup_i U_i=X$;
        \item for each $i,j$, $\varphi_i(U_i\cap U_j)$ is an open subset of a Banach space $E_i$ and $\varphi_i\colon U_i\to\varphi_i(U_i)$ is a bijection; 
        \item the map $\varphi_j\circ\varphi_i^{-1}\colon\varphi_i(U_i\cap U_j)\to\varphi_j(U_i\cap U_j)$ is a $C^p$-diffeomorphism for each pair of indices $i,j$.
    \end{enumerate}
    If such a $C^p$-atlas for $X$ exists, $X$ is equipped with the (unique) topology in which each $U_i$ is open and each $\varphi_i$ is a homeomorphism. The equivalence class of compatible\footnote{cf. \cite[p.~22]{lang1972}} atlasses is referred to as \emph{structure of a $C^p$-manifold} on $X$. If all $E_i$ are isomorphic, then there exists a Banach space $V$ and an atlas with $E_i=V$ for all $i$. In this case, $X$ is called \emph{$V$-Banach manifold}.  
\end{definition}

For the concepts of abstract tangent spaces, $C^m$-smooth mappings between manifolds and their derivatives (or ``\emph{tangents}''), see \cite[Definitions~3.3.3, 3.3.6 and~3.2.5]{abrahammarsdenratiu1988}.

\begin{theorem}[{\cite[Theorem~3.5.1]{abrahammarsdenratiu1988}}]\label{thm:inv-fct}
    Let $X,Y$ be Banach manifolds, $U\subset X$ open and $f\colon U\to Y$ a $C^p$ mapping where $p\in\N$. Assume that, for some point $x_0\in U$, the derivative $df_{x_0}\colon T_{x_0}X\to T_{f(x_0)}Y$ is an isomorphism. Then $f$ is a local $C^p$-diffeomorphism at $x_0$. 
\end{theorem}

In this article, whenever applying the abstract theory developed in \Cref{sec:abstr-theory}, we usually work on submanifolds of Banach spaces which are given by level-sets of submersions. For a definition of submanifolds, see \cite[Definition~3.2.1]{abrahammarsdenratiu1988}.

\begin{remark}
    Consider a Banach space $V$ and a submanifold $\mathcal{M}\subset V$. For $x\in \mathcal{M}$, we make the identification
    \begin{equation}
        T_x\mathcal{M}=\{\gamma'(0)\mid \exists\varepsilon>0,\gamma\in C^1((-\varepsilon,\varepsilon),V)\text{ with }\gamma((-\varepsilon,\varepsilon))\subset\mathcal{M}\text{ and }\gamma(0)=x\}.
    \end{equation}
\end{remark}

\begin{remark}\label{rem:on-submanifolds}
    Proceeding as in \cite[Example~2.10, Theorem~3.1 and Proposition~3.3]{rupp2020}, also see \cite[Theorem~3.5.4]{abrahammarsdenratiu1988}, one obtains the following. Consider Banach spaces $V,Y$, a $C^p$ map $\mathcal{B}\colon U\to Y$ on $U\subset V$ open, $x_0\in U$ such that $\mathcal{B}(x_0)=0$, $\mathcal{B}'(x_0)\colon V\to Y$ is surjective and write $V_0\vcentcolon=\ker\mathcal{B}'(x_0)$. Suppose that $V=V_0\oplus V_1$ for a closed subspace $V_1\subset V$. Then there exist open sets $\Omega_0\subset V_0$, $\Omega_1\subset V_1$ with $x_0\in\Omega=\Omega_0\times\Omega_1\subset U$ and a $C^p$ map $\psi\colon \Omega_0\to\Omega_1$ such that, for $\mathcal{M}\vcentcolon=\mathcal{B}^{-1}(0)\cap\Omega$,
    \begin{equation}
        \mathcal{M} = \{w+\psi(\omega):\omega\in\Omega_0\}\subset\Omega.
    \end{equation}
    Moreover, $\alpha\colon \Omega\to V$ with $\alpha(\omega+v_1)\vcentcolon=\omega+(v_1-\psi(\omega))$ for $\omega\in\Omega_0$ and $v\in V_1$ is a $C^p$-diffeomorphism onto its image with
    \begin{equation}
        \alpha(\mathcal{M})=\Omega_0\subset V_0\quad\text{and}\quad \alpha'(x_0)v_0=v_0\text{ for all $v_0\in V_0$}.
    \end{equation}
    In particular, $\mathcal{M}$ is a ($C^p$-regular) $V_0$-Banach manifold and $T_x\mathcal{M}=\ker\mathcal{B}'(x)$ for all $x\in \mathcal{M}$.
\end{remark}

\section{On analyticity of superposition operators}\label{sec:analycompo}

For the \L ojasiewicz-Simon inequality we need that the mappings involved are analytic. Here we shortly present a useful result to prove such regularity for compositions of mappings. We follow \cite[Chapter II, \textsection 5]{tullio1988}.

\begin{definition}
    Let $\Omega \subset \R^n$ be an open and bounded domain and $U \subset \R^N$ be open. Let $f=f(x,y)\colon \Omega \times U \to \R$, be smooth on $\overline{\Omega} \times U$. Then,
    \begin{enumerate}[(a)]
        \item $f$ is \emph{analytic in $y$ at $y_0 \in U$ uniformly with respect to $x$} if for each $x_0 \in \overline{\Omega}$ there is a neighborhood $U_0 \subset \overline{\Omega} \times U$ of $(x_0,y_0)$ such that
        $$f(x,y) = \sum_{k=0}^{\infty} \frac{1}{k!} \partial^{(k)}_y f (x,y_0)  (y-y_0)^k,$$
        for all $(x,y) \in U_0$; 
        \item $f$ is \emph{analytic in $y$ uniformly with respect to $x$} if for each $y_0 \in U$,  $f$ is analytic in $y$ at $y_0$ uniformly with respect to $x$.
        \item a vector valued map $\vec{f}: \Omega \times U \to \R^{M}$, $\vec{f}(x,y)$, that is $C^{\infty}$ on $\overline{\Omega} \times U$ is   \emph{analytic in $y$ uniformly with respect to $x$} if each of its $M$ component is analytic in $y$ uniformly with respect to $x$.
    \end{enumerate}
\end{definition}

\begin{remark}\label{rem:B2} With the same notation as above, one finds that  $f$ is analytic in $y$ uniformly with respect to $x$ if and only if, for every compact $K \subset U$ there is a constant $C_K > 0$ such that
\begin{equation}
|\partial_y^{\beta}f(x,y)| \leq C_K^{1+|\beta|} \beta! \mbox{ for all }(x,y) \in \overline{\Omega} \times K, \beta \in \N_0^N  .
\end{equation}
\end{remark}

\begin{theorem}\label{thm:analytic_composition}
 Let $U\subset \R^N$ be open and $\Omega \subset \R^n$ be a bounded domain with the cone property.  
 Let $M \in \N$, $m \in \N$, $p \in (1,\infty)$ and $r \in \N_0$ such that $(m+r)p>n$. 
 
 Assume that $f \in C^{\infty}(\overline{\Omega} \times U, \R^M)$ and that for all $\alpha \in \N_0^n$ with $|\alpha|\leq m$, the mappings $(x, y) \mapsto D^{\alpha}_x f(x,y)$,  are analytic in $y$ uniformly with respect to $x$. Then,  the superposition operator
 $$ F: W^{m+r,p}(\Omega,\R^N) \to  W^{m,p}(\Omega,\R^M), \quad F(\sigma)( x ) = f(x,\sigma(x)), \, x \in \Omega,$$
 is analytic.
\end{theorem}

\begin{lemma}\label{lem:analit-local}
    Let $\Sigma$ be a compact connected smooth surface with or without boundary. Let $U\subset \R^N$ be open and let $f\in C^\infty(\Sigma\times U,\R^M)$. Let $m,n\in\N_0$, $p,q\in [1,\infty]$. Then the superposition operator
    $$ F: W^{n,q}(\Sigma,\R^N) \to  W^{m,p}(\Sigma,\R^M), \quad F(\sigma)( x ) = f(x,\sigma(x)), \, x \in \Sigma$$
    is analytic if and only if there exists a finite covering of $\Sigma$ by local inverse charts $\psi_j\colon \Omega_j\overset{\approx}{\to}\psi_j(\Omega_j)\subset\Sigma$, $\Omega_j\subset \R\times[0,\infty)$ open
    with $\psi_j\in C^\infty(\bar \Omega_j,\Sigma)$ for all $j\in J$, such that the superposition operator
    $$G_j\colon W^{n,q}(\Omega_j,\R^N) \to  W^{m,p}(\Omega_j,\R^M),  \quad G_j(\tilde\sigma)(x)  = f(\psi_j(x),\tilde\sigma(x))$$
    is analytic for each $j\in J$.
\end{lemma}
\begin{proof}
For the `if' part of the statement, let $\chi_j$, $j\in J$, be a partition of unity, subordinate to $\psi_j(\Omega_j)$. We write
\begin{align}\label{eq:analytic_decomposition}
F(\sigma) = \sum_{j\in J} \chi_j f(\cdot, \sigma) = \sum_{j\in J}[\chi_j\circ \psi_j f(\psi_j, \sigma\circ\psi_j)]\circ\psi_j^{-1}.
\end{align}
Since $\psi_j\in C^\infty(\bar \Omega,\Sigma)$, the map
\begin{align}
W^{n,q}(\Sigma,\R^N)\to W^{n,q}(\Omega_j, \R^N), \sigma\mapsto \sigma\circ \psi_j
\end{align}
is linear and bounded, hence analytic. Hence, by assumption,
\begin{align}
W^{n,q}(\Sigma,\R^N) \to W^{m,p}(\Omega_j,\R^M), \sigma\mapsto \chi_j\circ\psi_j f(\psi_j, \sigma\circ\psi_j)
\end{align}
is analytic as a composition of analytic maps. 
Arguing similarly as above, we find that
\begin{align}
W^{n,q}(\Sigma,\R^N) \to W^{m,p}_c(\psi_j(\Omega_j),\R^M), \sigma\mapsto [\chi_j\circ\psi_j f(\psi_j, \sigma\circ\psi_j)]\circ\psi_j^{-1}
\end{align}
is analytic, where $W^{m,p}_c(\psi_j(\Omega_j),\R^M)$ is the closure in $W^{m,p}(\psi_j(\Omega_j),\R^M)$ of all smooth functions $u\in C^\infty(\Sigma,\R^M)$ whose support is contained in $\psi_j(\Omega_j)$. The extensions by zero
\begin{align}
W^{m,p}_c(\psi_j(\Omega_j),\R^M)\to W^{m,p}(\Sigma,\R^M)
\end{align}
are linear and bounded, hence analytic. Finally, taking the finite sum over $j\in J$, the claim follows from \eqref{eq:analytic_decomposition}.

For the `only if' part, the existence of a finite covering as in the statement is standard. For any $j\in J$, choose bounded and linear extension and restriction operators \linebreak $E_j\colon W^{n,q}(\psi_j(\Omega_j),\R^N) \to W^{n,q}(\Sigma,\R^N)$ and $R_j\colon W^{m,p}(\Sigma,\R^M)\to W^{m,p}(\psi_j(\Omega_j),\R^M)$. By the assumption, we find that
\begin{align}
    K_j \vcentcolon = R_j\circ F\circ E_j\colon  W^{n,q}(\psi_j(\Omega_j),\R^N) \to W^{m,p}(\psi_j(\Omega_j),\R^M)
\end{align}
is analytic. Arguing as above, we conclude that
\begin{align}
 W^{n,q}(\Omega_j,\R^N) \to W^{m,p}(\Omega_j,\R^M), \tilde\sigma \mapsto K_j(\tilde\sigma\circ \psi_j^{-1})\circ \psi_j
\end{align}
is analytic. This is exactly the map $G_j$ yielding the assertion.
\end{proof}

\section{Technical proofs}\label{app:techproofs}

\begin{proof}[Proof of \eqref{eq:deltaE} and \eqref{eq:loja-ass-1}]
    Without loss of generality, we may suppose that $w$ and $\varphi$ are smooth. Since $\varphi\in T_{w}\mathcal{M}$, there exists $\gamma\colon(-\varepsilon,\varepsilon)\to\mathcal{M}$ smooth with $\gamma(0)=w$ and $\gamma'(0)=\varphi$. We compute
    \begin{equation}
        \partial_{t,0}f_{\gamma(t)}=\partial_{t,0}\Phi(\cdot,\gamma(t))=\varphi\cdot \xi\circ f_w=\vcentcolon\varphi^{\perp}\nu_{f_w}+df_w.\varphi^{\top}
    \end{equation}
    where $\varphi^{\perp}=\varphi\langle\xi\circ f_w,\nu_{f_w}\rangle$. Moreover, using \Cref{lem:def-fw}, $f_{w+t\varphi}(\partial \Sigma)\subset S$ for $|t|>0$ sufficiently small. With \eqref{eq:abldistance}, 
    one finds on $\partial \Sigma$
    \begin{align}
        0&=\partial_{t,0} d^S(f_{w+t\varphi}) = \langle N^S\circ f_w,\varphi\cdot\xi\circ f_w\rangle  =  \langle df_w.\eta_{f_w} , df_w .\varphi^{\top}\rangle = \langle \eta_{f_w},\varphi^{\top}\rangle_g, \label{eq:der-B-3}
    \end{align}  
    using that $N^S\circ f_w= \frac{\partial f_w}{\partial\eta_{f_w}}$ on $\partial \Sigma$. For the same reason, with $\tau_{f_w}$ as in the beginning of \Cref{sec:loja_Willmore}, on $\partial\Sigma$
    \begin{align}
        0&=\frac{\partial}{\partial\tau_{f_w}} \langle \nu_{f_w},N^S\circ f_w\rangle  =  \langle \frac{\partial}{\partial\tau_{f_w}} \nu_{f_w},\frac{\partial f_w}{\partial \eta_{f_{w}}}\rangle + \langle \nu_{f_w},\frac{\partial}{\partial \tau_{f_w}} (N^S\circ f_w)\rangle \\
        &= - A(\tau_{f_w},\eta_{f_{w}}) - A^S(\nu_{f_w},\frac{\partial f_w}{\partial\tau_{f_w}}).\label{eq:der-B-4}
    \end{align}
    Proceeding as on \cite[p.~6, after Equation (2.14)]{alessandronikuwert2016}, again using $\frac{\partial f_w}{\partial\eta_{f_w}}= N^S\circ f_w$, one now computes on $\partial \Sigma$
    \begin{align}
        0=\partial_{t,0} \langle \nu_{f_{\gamma(t)}},N^S\circ f_{\gamma(t)}\rangle = - \frac{\partial \varphi^{\perp}}{\partial\eta_{f_w}} - \varphi^{\perp}A^S(\nu_{f_w},\nu_{f_w}) -  A(\varphi^{\top},\eta_{f_w}) - A^S(\nu_{f_w},df_w.\varphi^{\top}).
    \end{align}
    By \eqref{eq:der-B-3}, $\varphi^{\top}$ is parallel to $\tau$ and thus, using \eqref{eq:der-B-4}, $0= A(\varphi^{\top},\eta_{f_w}) + A^S(\nu_{f_w},df_w.\varphi^{\top})$ so that
    \begin{equation}\label{eq:der-B-1}
        0= \frac{\partial \varphi^{\perp}}{\partial\eta_{f_w}} + \varphi^{\perp}A^S(\nu_{f_w},\nu_{f_w}).
    \end{equation}    
    Using \cite[Theorem~1]{alessandronikuwert2016},  
    \begin{align}
        2\E'(w)\varphi &= 2\partial_{t,0}\W(f_{w+t\varphi}) \\
        &= \int_\Sigma \delta\E(w) \varphi^{\perp} \langle \xi\circ f_w,\nu_{f_w}\rangle \dd\mu_{f_w} 
        \\
        &\qquad - \int_{\partial \Sigma} \Big( \varphi^{\perp}\frac{\partial H_{f_w}}{\partial\eta_{f_w}} - \frac{\partial \varphi^{\perp}}{\partial\eta_{f_w}}H_{f_w}
 \Big) - \frac12 H_{f_w}^2\langle \varphi^{\top},\eta_{f_w}\rangle_g\dd s.
    \end{align}
    By \eqref{eq:der-B-3}, $\langle \varphi^{\top},\eta_{f_w}\rangle_g=0$ on $\partial \Sigma$. Moreover, using \eqref{eq:der-B-1} and \eqref{intro-eq:fbc}, one also finds
    \begin{equation}
         \varphi^{\perp}\frac{\partial H_{f_w}}{\partial\eta_{f_w}} - \frac{\partial \varphi^{\perp}}{\partial\eta_{f_w}}H_{f_w} = 0\quad\text{on $\partial \Sigma$}.
    \end{equation}
    Altogether, \eqref{eq:deltaE} is proved. Then \eqref{eq:loja-ass-1} is immediate choosing $a_2$ sufficiently small such that all metrics $g_{f_w}$ are uniformly equivalent to $g_{\bar{f}}$ for $w\in U_2$ by the Sobolev embedding $W^{4,2}(\Sigma)\hookrightarrow C^1(\Sigma)$, and using \eqref{eq:estimate-xi0fw-in-nu-fw}.
\end{proof}

\subsection{Proof of {\Cref{prop:reparam-as-fw}}}\label{app:proof-prop:reparam-as-fw}

Using the classification of surfaces \cite[Chapter~9, Theorem~3.7]{Hirsch} and Nash's embedding theorem, we can assume that there is a closed embedded surface $\hat\Sigma\subseteq\R^M$ such that $\Sigma$ is a submanifold of $\hat\Sigma$. By \cite[Theorem~1 in Sect. 2.12.3]{Simon1996}, there exist tubular neighborhoods $U_{\partial\Sigma}$ and $U_{\hat\Sigma}$ of the closed submanifolds $\partial\Sigma$ resp.\ $\hat\Sigma$ in $\R^M$ such that the associated nearest point projections $\Pi^{\partial\Sigma}\colon U_{\partial\Sigma}\to\partial\Sigma$ resp.\ $\Pi^{\hat\Sigma}\colon U_{\hat\Sigma}\to\hat\Sigma$ are smooth. Moreover, for $x\in\partial\Sigma$, $d\Pi^{\partial\Sigma}_x\colon \R^M\to T_x\partial\Sigma$ is the orthogonal projection onto $T_x\partial\Sigma$.

\begin{lemma}\label{lem:D-submanifold}
    There exists $\varepsilon>0$ such that
    \begin{align}
        \mathcal{R} &\vcentcolon = \{\zeta\in W^{4,2}(\Sigma,\R^M):\|\zeta\|_{W^{4,2}(\Sigma)}<\varepsilon , 
        x+\zeta(x)\in\partial\Sigma\text{ for $x\in\partial \Sigma$}\}
    \end{align}
    
    is a Banach manifold, and $\psi_\zeta(x)\vcentcolon=\Pi^{\hat\Sigma}(x+\zeta(x))$ for $x\in\Sigma$ defines a $C^1$-diffeomorphism $\psi_\zeta\colon\Sigma\to\Sigma$ for each $\zeta\in\mathcal{R}$. Furthermore, 
    \begin{equation}\label{eq:D-tan-space}
        T_0\mathcal{R} = \{\zeta\in W^{4,2}(\Sigma,\R^M):\zeta(x)\in T_x\partial\Sigma\text{ for all $x\in\partial\Sigma$}\}.
    \end{equation}
\end{lemma}
\begin{proof} 
Since $W^{4,2}(\Sigma)\hookrightarrow C^1(\Sigma)$, we obtain $\|\psi_\zeta-\mathrm{id}_\Sigma\|_{C^1(\Sigma)}\leq C(\hat\Sigma)\varepsilon$. Choosing $\varepsilon>0$ sufficiently small, we thus have that $\psi_\zeta\colon\Sigma\to\hat\Sigma$ is a diffeomorphism onto its image. Moreover, using $\zeta\in\mathcal R$ and \cite[Proposition 5.11]{MR583436}, we have $\psi_\zeta(\partial\Sigma)=\partial\Sigma$. As $\psi_\zeta$ is a diffeomorphism, since $\Sigma$ is connected and as $\psi_\zeta(\Sigma\setminus\partial\Sigma)\subset\hat\Sigma\setminus\partial\Sigma$, we have either $\psi_\zeta(\Sigma\setminus\partial\Sigma)\subset \Sigma\setminus\partial\Sigma$ or $\psi_\zeta(\Sigma\setminus\partial\Sigma)\subset \hat\Sigma\setminus\Sigma$. Choosing $\varepsilon>0$ sufficiently small, we can exclude the latter option. Altogether, choosing $\varepsilon>0$ sufficiently small, for each $\zeta\in\mathcal R$, we have that $\psi_\zeta\colon\Sigma\to\Sigma$ is a $C^1$-diffeomorphism.  

It remains to show that, choosing $\varepsilon>0$ small enough, $\mathcal R$ is a Banach manifold with \eqref{eq:D-tan-space}. Consider 
    \begin{align}
        &G\colon B_{\varepsilon}(0)\subset W^{4,2}(\Sigma,\R^M)\to \{h\in W^{\frac{7}{2},2}(\partial \Sigma,\R^M)\mid h(x)\perp T_x\partial\Sigma\text{ for all }x\in\partial\Sigma\},\\
        &G(\zeta)(x)\vcentcolon=(\mathrm{Id}-d\Pi_x^{\partial\Sigma}).(x+\zeta(x)-\Pi^{\partial\Sigma}(x+\zeta(x))).
    \end{align}   
    One can verify that, choosing $\varepsilon$ sufficiently small, we have for $\zeta\in B_\varepsilon(0)$ that $G(\zeta)=0$ if and only if $x + \zeta(x)\in\partial\Sigma$ for all $x\in\partial\Sigma$. Moreover, $G$ is continuously differentiable with
    \begin{equation}
        G'(0).\varphi(x) = (\mathrm{Id}-d\Pi_x^{\partial\Sigma}).\varphi(x) \qquad \mbox{ for }\varphi \in W^{4,2}(\Sigma,\R^M).
    \end{equation}
    Thus, $G$ is a submersion 
    in $\zeta=0$ with
    \begin{equation}
        \ker G'(0) = \{\zeta\in W^{4,2}(\Sigma,\R^M):\zeta(x)\in T_x\partial\Sigma\text{ for all $x\in\partial\Sigma$}\}.
    \end{equation}
    
    Altogether, the statement follows from \Cref{rem:on-submanifolds} with $\mathcal R=G^{-1}(\{0\})$ after potentially further shrinking $\varepsilon$.
\end{proof}

\begin{lemma}\label{lem:N-submanifold}
    For a suitably small neighborhood $U$ of $\bar{f}$ in $W^{4,2}(\Sigma,\R^3)$, the set $\mathcal{N}\vcentcolon=\{f\in U\mid f(\partial \Sigma)\subset S\}$ is a Banach manifold of immersions. Furthermore, 
    \begin{equation}\label{eq:N-submanifold}
        T_{\bar{f}}\mathcal{N} = \{h\in W^{4,2}(\Sigma,\R^3):\langle N^S\circ\bar{f},h\rangle\big|_{\partial \Sigma}=0\}.
    \end{equation}
\end{lemma}
\begin{proof}
    The map $F\colon U\to W^{\frac{7}{2},2}(\partial \Sigma,\R)$ with $F(f)\vcentcolon= d^S(f)\big|_{\partial \Sigma}$ satisfies $F'(\bar{f})(v)= \left.\langle N^S \circ \bar{f}, v \rangle \right|_{\partial \Sigma}$ and is a submersion\footnote{For $h \in W^{\frac{7}{2},2}(\partial \Sigma,\R)$ consider $\tilde{v} \in W^{4,2}(\Sigma)$ such that $\left.\tilde{v}\right|_{\partial \Sigma}=h$ and $v\vcentcolon=\chi N^S(\Pi^S \circ \bar{f}) \tilde{v}$ with $\chi$ a smooth function with support in  a neighborhood of $\partial \Sigma$ and identically equal to $1$ on $\partial \Sigma$. Then, $F'(\bar{f})(v)=h$.} in $f=\bar{f}$ with
    \begin{equation}
       \ker F'(\bar f) = \{h\in W^{4,2}(\Sigma,\R^3):\langle N^S\circ\bar{f},h\rangle\big|_{\partial\Sigma}=0\}.
    \end{equation}
    The claim then follows from \Cref{rem:on-submanifolds}.
\end{proof}
\begin{proof}[Proof of \Cref{prop:reparam-as-fw}.]
    After potentially shrinking $\varepsilon$ and $\tilde a$, consider 
    \begin{align}
        &F\colon\{w\in W^{4,2}(\Sigma,\R):\|w\|_{W^{4,2}(\Sigma)}<\tilde a\} \times \mathcal{R} \to \mathcal{N},\\
        &F(w,\zeta)(x)\vcentcolon= f_w(\psi_\zeta(x))=\Phi(\psi_\zeta(x),w(\psi_\zeta(x))). 
    \end{align}
   Using \Cref{prop:ex-gauss-coord}, $F$ is a $C^1$-map between manifolds. For $\varphi\in W^{4,2}(\Sigma,\R)$ and $\zeta\in \mathcal{R}$, using $F(0,0)=\bar{f}$ and $\partial_r\Phi=\xi$ (see \Cref{prop:ex-gauss-coord}), its differential $dF_{(0,0)}\colon W^{4,2}(\Sigma,\R)\times T_0\mathcal{R}\to T_{\bar{f}}\mathcal{N}$ satisfies
    \begin{equation}
        dF_{(0,0)}(\varphi,\zeta) = \varphi \cdot \xi(\cdot,0) + d\bar{f}. \zeta. 
    \end{equation}
    
    We argue now that $dF_{(0,0)}$ is invertible. Injectivity is clear since $\xi(x,0)\notin d\bar f_x(T_x\Sigma)$. To prove surjectivity, consider $h \in T_{\bar{f}}\mathcal{N}$ and define $\varphi\colon\Sigma \to \R$,
    $$ \varphi(x)= \frac{\langle h(x), \nu_{\bar f}(x)\rangle}{\langle \xi (x,0), \nu_{\bar f}(x) \rangle}. $$
    Then $\varphi\in W^{4,2}(\Sigma,\R)$. Moreover, 
    $
        \langle \nu_{\bar f}(x), h(x)-\varphi(x)\xi(x,0)\rangle = 0
    $ for all $x\in\Sigma$. Thus, using that $\bar f$ is an immersion and that $\Sigma$ is compact, there exists a unique $\zeta\in W^{4,2}(\Sigma,\R^M)$ with $\zeta(x)\in T_x\Sigma$ for all $x\in\Sigma$ such that
    $$
        d\bar f_x.\zeta(x)=h(x)-\varphi(x)\xi(x,0)\quad \text{for all $x\in\Sigma$}.
    $$
    Once we have verified that $\zeta\in T_0\mathcal R$, this choice yields that $(dF_{(0,0)}).(\varphi,\zeta)=h$ and thus surjectivity. 
   On $\partial \Sigma$, using $N^S(\bar{f}(x))=\partial_{\eta_{\bar f}}\bar{f}(x)$, $\xi(\cdot,0)=\nu_{\bar f}$ and \eqref{eq:N-submanifold},
    \begin{align}
        0&=\langle N^S(\bar{f}(x)),h(x)\rangle = \varphi(x) \langle \partial_{\eta_{\bar f}}\bar{f}(x),\nu_{\bar{f}}(x)\rangle + \langle d\bar{f}_x.\eta_{\bar f}(x),d\bar{f}_x.\zeta(x)\rangle \\
        &= g_{\bar{f}}(\eta_{\bar f}(x),\zeta(x)).
    \end{align}
    Thus, $\zeta(x)$ is parallel to $\tau(x)\in T_x\partial \Sigma$ which yields $\zeta\in T_0\mathcal{R}$, using \eqref{eq:D-tan-space}.
    
    The local inverse mapping theorem in manifolds, see \Cref{thm:inv-fct}, yields the claim.
\end{proof}

\begin{remark}\label{rem:reparam_fw_C4a}
    Note that \Cref{prop:reparam-as-fw} remains valid when the Sobolev spaces $W^{4,2}$ are replaced by the Hölder spaces $C^{4,\alpha}$. In this case, the trace spaces appearing in the above proof need to be suitably adapted.
\end{remark}

\bibliographystyle{abbrv}
\bibliography{biblio}

\end{document}